\documentclass[reqno,11pt]{amsart}
\usepackage{graphicx}
\usepackage{geometry}
\usepackage{amsthm,amsmath,verbatim,amssymb}
\usepackage{indentfirst}
\usepackage{geometry}
\usepackage{setspace}
\usepackage{graphicx}
\usepackage{tabularx}
\usepackage{booktabs}
\usepackage{tikz}
\usetikzlibrary{matrix, positioning, shapes.geometric, decorations.pathmorphing}
\usepackage{enumerate}
\usepackage{enumitem}
\usepackage{amsmath}
\usepackage{bbm}
\usepackage{caption}
\usepackage{bm}
\allowdisplaybreaks[4]
\usepackage{hyperref}
\hypersetup{hidelinks} 

\newcommand{\abs}[1]{\left\vert#1\right\vert}
\newcommand{\norm}[1]{\left\Vert#1\right\Vert}

\renewcommand{\i}{\mathrm{i}}
\renewcommand{\mod}{\text{mod }}

\newcommand{\E}{\mathbb{E}}

\renewcommand{\bar}{\overline}

\newtheoremstyle{mystyle}{3.5pt}{3.5pt}{\upshape}{0cm}{\bfseries\rmfamily}{}{1em}{} %定义定理的环境，去掉斜体
\theoremstyle{mystyle}
\newtheorem{thm}{Theorem}[section] 
  
\newtheorem{lem}[thm]{Lemma} 

\newtheorem{cor}[thm]{Corollary}

\newcommand{\pfa}{\text{Pf }}
\newcommand{\pfaD}{\text{Pf}_D}
\newcommand{\pfaO}{\text{Pf}_O}
\newcommand{\M}{\mathbf{M}}
\newcommand{\U}{\mathbf{U}}
\newcommand{\ord}{\text{Ord}}
\newcommand{\semi}{\text{sc}}

\newcommand{\ep}{\epsilon}
\renewcommand{\tilde}{\widetilde}

\numberwithin{equation}{section}  

\begin{document}
	\title{Small gaps of GSE}
	\author[Feng]{Renjie Feng}

	\address{Sydney Mathematical Research Institute, The University of Sydney, Australia.}
	
	\email{renjie.feng@sydney.edu.au}
	
	\author[Li]{Jiaming Li}
	\address{LSEC, Academy of Mathematics and Systems Science, Chinese Academy of Sciences, Beijing,  China,100190.}
	\email{jiamingli19960730@163.com}
	\author[Yao]{Dong Yao}
	\address{Research Institute of Mathematical Sciences, Jiangsu Normal University, Xuzhou, China, 221116.}

	\email{dongyao@jsnu.edu.cn}
	
	\date{\today}

	\begin{abstract}
		In this paper, we study the smallest gaps for the Gaussian symplectic ensemble (GSE). We prove that the rescaled smallest gaps and their locations converge to a Poisson point process with an explicit rate. The approach provides an alternative proof for the GOE case and complements the results in \cite{FTW}.  By combining the main results from \cite{BB, FTW, FW2}, the study of the smallest gaps for the classical random matrix ensembles C$\beta$E and G$\beta$E for $\beta = 1, 2,$ and $4$ is now complete.   \end{abstract}
	
	\maketitle
	
	\section{Introduction}\label{introd}
	
	In random matrix theory, one of the main concerns is the distributions of eigenvalues and the gaps between consecutive eigenvalues. Back in the 1950s, Wigner predicted that the spacings between the spectrum of a heavy atom's nucleus should resemble the spacings between the eigenvalues of certain Gaussian random matrices. This prediction is confirmed by experiments in  nuclear physics \cite{M}. Wigner also predicted the distribution of average gaps, which follow the Gaudin-Mehta distribution \cite{Deift_1, M}. In fact, the Gaudin-Mehta distribution has been proved to be true even for a single gap in the bulk of the semicircle law of GUE \cite{T} and GOE \cite{ EY}, which is universal for more general random matrices.  Other than the average gaps and the single gap, another fundamental quantity of interest is the extreme spacings. 
	
	%So far, the asymptotic distributions of the single eigenvalue gap  of classical Gaussian  Orthogonal/Unitary/Symplectic  Ensembles (often denoted by G$\beta$E with $\beta$=1,2,4) 
	%and their universalities for 
	%Wigner matrices and general $\beta$-ensembles have been well understood  \cite{M, BEY,ey,  T, T2, Sh, BFG}. The limiting laws are called Gaudin-Mehta distributions \cite{M}, which depend on the specific symmetry class (the value of $\beta$).   
	%There are also a few results on the extreme gaps (the smallest/largest eigenvalue gap) \cite{BB, B2,GOE, FW,FW2,V,So}. 
	
	%Actually,  	eigenvalue gaps also have deep connections with many other areas. For instance, the behaviors of extreme gaps of GUE random matrix share striking similarlities with the gaps of Riemann zeta zeros \cite[Section 1.3]{BB}.  
	
	%On the other hand, for the distributions of extreme eigenvalue gaps of large random matrices,  the picture is much less complete.

	%In this work we focus on the small gap problem for GOE and GSE. We now give a quick review of existing results on the small gaps of random matrices before stating ours.

	To state the main results, we first recall two types of point processes studied intensively in random matrix theory.  The first is the Gaussian $\beta$-ensemble (G$\beta$E) for $\beta>0$: Given $n$ points $\lambda_1, \cdots ,  \lambda_n  $   on $\mathbb R$ with  the joint density
	\begin{equation}	\label{jpdf of Gauss ensemble}J(\lambda_1 , \cdots, \lambda_n )=\frac{1}{Z_{\beta, n}} {\prod_{k=1}^n e^{-\frac{\beta n}{4}\lambda_k^2}}  {\prod_{i<j}\left|\lambda_j-\lambda_i\right|^\beta}, \end{equation}
	% \\ =&\frac{1}{Z_{\beta, n}}e^{ -\beta n(-\frac 1 2\sum _{i\neq j}\log |\lambda_i-\lambda_j|+ \frac 1 4\sum \lambda^2_i)}  \end{aligned}$$
	where by the Selberg integral, $$Z_{\beta, n}=(2\pi)^{n/2}\left(\frac{\beta n}2\right)^{-\frac{n(n-1)\beta}{4}-\frac n 2}\prod_{j=1}^n\frac{\Gamma(1+j\frac{\beta}{2})}{\Gamma(1+\frac{\beta}2)}.$$ In particular, $\beta=1, 2$ and $4$ correspond to the joint density of eigenvalues of the classical random matrices of GOE, GUE and GSE, respectively.	A remarkable fact is that GUE is a determinantal point process, while GOE and GSE are Pfaffian point processes.   The limit of the global distribution of the point processes of G$\beta$E is given by the Wigner semicircle law   as $n\to\infty$:   \begin{equation}\label{semicirclelawa}
		\rho_{\semi}(x) = \dfrac{1}{2\pi}\sqrt{4-x^2},\quad x \in [-2,2]. 
	\end{equation}   
	Another point process is the circular $\beta$-ensemble (C$\beta$E) for $\beta>0$:  Given $n$ points $e^{i\theta_1}, \cdots ,  e^{i\theta_n}  $  on $\mathbb{S}^1$,  the eigenangles $\theta_1, \cdots ,  \theta_n $ have the joint density 
	
	$$J(\theta_1 , \cdots, \theta_n )=\frac{1}{C_{\beta, n}}  \prod_{i<j}\left|e^{i\theta_j}-e^{i\theta_i}\right|^\beta, $$
	where by the Selberg integral again, $$C_{\beta, n}=(2\pi)^{n}\frac{\Gamma(1+\beta n/2)}{(\Gamma(1+\beta /2))^n}.$$ In particular, $\beta=1, 2$ and $4$ correspond to the joint density of eigenangles of classical random matrices of COE,  CUE and CSE, respectively. Similarly, CUE has a determinantal structure, while COE and CSE have Pfaffian structures.  The global distribution of the eigenangles of C$\beta$E tends to the uniform measure on $\mathbb{S}^1$ as $n\to\infty$.

	%For both the circular ensemble and Gaussian ensemble, a very critical fact is that the point processes have determinantal structures for $\beta=2$, and have Pfaffian structures for $\beta=1$ and $4$. These structures make the study of the 
	We first summarize the existing results on the smallest gaps for both the circular ensembles and Gaussian ensembles. In \cite{V}, 
	Vinson  first obtained  limiting distributions of rescaled smallest gap of CUE and GUE. In \cite{So},  Soshnikov  studied  the smallest gaps of determinantal point processes with translation invariant kernels.  In \cite{BB},  Ben Arous-Bourgade employed Soshnikov's method  to further derive the limiting joint density of $k$ smallest gaps of CUE and GUE for  $k\geq 1$, along with the locations where these gaps occur. Note that the proofs in \cite{BB, So, V} rely heavily on the determinantal structures.  %For the determinantal point processes,  the correlation function is the determinant of a symmetric kernel matrix. Due to the rapid decay of the kernel, e.g., GUE,  the leading order terms appear on the diagonal $2\times 2$ blocks. The results then follow from the estimates of the determinant by the Hadamard-Fischer inequality. 
	
	In \cite{FTW},  the authors derived the smallest gaps for GOE, addressing the problem beyond the determinantal structures. Note that \cite{FTW} only proves the limiting distribution of the smallest gaps of GOE, without providing information on their locations.  The  intuitive idea behind the method in \cite{FTW} is based on the following observation from a statistical physics perspective: for a pair of two particles with charge 1 forming the smallest gap, these two particles will stick together to form one `double particle' with charge 2 in the limit. That is, the original system of one-component log-gas will become a new system of two-component log-gas. To prove the main results, it is necessary to study the limit of the ratios of the partition functions of these two systems, which can be analyzed using Selberg-type integrals.  In \cite{FW2}, using the same method, the authors derived the smallest gaps of C$\beta$E with integer-valued $\beta$, with results holding particularly for the classical random matrices of COE, CUE, and CSE.
	
	However, one of the main obstacles with this method is that the Selberg-type integrals are intractable in most cases, such as GSE. In the current paper,  we derive the smallest gaps of GSE and the locations where they occur. Our approaches relies on combinatorial arguments to bound certain quantities arising from the Pfaffian structure. We expect our computations  to provide new insights into the study of various Pfaffian point processes. In particular, it provides an alternative proof for  GOE, where the previously omitted location information will be included.  This, together with the main results in \cite{BB, FTW, FW2}, will complete the study of the smallest gaps and their locations for the classical random matrices of C$\beta$E and G$\beta$E with $\beta=1, 2$ and $4$. 
	
	Regarding the largest gaps, only a few results are known. In \cite{BB}, Ben Arous-Bourgade determined the decay order of the largest gaps for both the CUE and GUE (within the interior of the semicircle law). Additionally, \cite{FW} proves that these largest gaps converge to Poisson point processes. Consequently, the fluctuations of the largest gaps are shown to follow Gumbel distributions in the limit. Again, the determinantal structures of the CUE and GUE play a crucial role in the analyses in \cite{BB, FW}. It is worth noting that the decay orders and fluctuations of the largest gaps for other ensembles, such as the COE and CSE, remain unknown.

	The works \cite{B2, LLM}  proved that the results in \cite{BB, FTW, FW}  are universal, 
	i.e., they hold for  Wigner ensembles with general distributions.  Figalli-Guionnet  also derived the smallest gaps for a several-matrix model in \cite{FG}. %where multiple  invariant ensembles  interact through a potential. 

	\subsection{Main results}
	%As just mentioned, in this paper we focus on Gaussian-$\beta$ ensembles with eigenvalue p.d.f.s given in \eqref{jpdf of Gauss ensemble}.
	%	We now describe the result of \cite{GOE} in detail. 
	Let $\lambda_1<\lambda_2< \cdots<\lambda_n$ be sampled from G$\beta$E with joint density given in  \eqref{jpdf of Gauss ensemble}. For any fixed small $\epsilon>0$, consider the two-dimensional point process
	$$
	\label{gaps process}
	\Upsilon^{(\beta)}_{n} = \sum_{i=1}^{n-1}\delta_{\left(n^{\frac{\beta +2}{\beta +1}}(\lambda_{i+1}-\lambda_{i}),\,\,\, \lambda_i\right)}\cdot \mathbbm{1}[\lambda_i\in(-2+\ep,2-\ep)].
	$$
	% (Here the restriction to $(-2+\ep,2-\ep)$ is due to technical reasons arising from Plancherel-Rotach asymptotics, and should not be essential for the small gaps problem.)%As mentioned before, for the $\beta=2$ case (GUE), Ben Arous and Bourgade has proved the convergence of  $\Upsilon_n^{(\beta)}$ to a  (non-homogeneous) Poisson point process. In this work we focus on GSE and GOE. Let
	%$$
	%\rho_{\semi}(x) = \dfrac{1}{2\pi}\sqrt{4-x^2}  \mathbbm{1}[x \in (-2,2)]
	%$$
	%be the semicircle law.\
	The main result is 
	\begin{thm}
		\label{small gaps of GbetaE}
		For $\beta = 1,2$ and $4$, $\Upsilon^{(\beta)}_{n}$ converges to a Poisson point process $\Upsilon^{(\beta)}$ on $(0,+\infty)\times (-2+\ep,2-\ep)$ with intensity
		$$
		\mathbb{E}\Upsilon^{(\beta)}(A\times I)=\left(\dfrac{1}{c_{\beta}}\int_A u^{\beta}\mathrm{d}u\right)\int_I\left(2\pi\rho_{\semi}(x)\right)^{\beta+2}\mathrm{d}x,
		\label{definition of the poisson}
		$$
		where $\rho_{sc}$ is the semicircle law \eqref{semicirclelawa}, $A$ is any bounded measurable set in $(0,+\infty)$, $I$ is any interval in $(-2+\ep,2-\ep)$,  and 
		\begin{align*}
			c_1 = {48\pi},\quad
			c_2 = {48\pi^2},\quad
			c_4 = {540\pi^2}.
		\end{align*}
	\end{thm} 
	As we mentioned earlier, the case for $\beta=2$ has been proven in \cite{BB}. For $\beta=1$,  \cite{FTW} only studies the 1-dimensional point process of the smallest gaps $\sum_{i=1}^{n-1}\delta_{n^{{3/2}}(\lambda_{i+1}-\lambda_{i})}$, where the location information is missing. In this paper, we prove Theorem \ref{small gaps of GbetaE} for $\beta=4$ in detail, and then sketch the proof for GOE.  %We also comment that \cite[Theorem 1.2 (ii)]{B2} in fact needs the full power of such a statement for comparisons with GOE case.

	%	For $\beta=1$,  take $I=(-2,2)$ (ignoring the restriction to $(-2+\ep, 2-\ep)$) and consider the first margin of $\Upsilon^{(1)}$,  one  obtains exactly \cite[Theorem 1]{GOE}. 
	%In other words, compared to \cite{GOE}, the GOE case of Theorem \ref{small gaps of GbetaE} contains the additional information of the locations where small gaps occur.
	
	%\footnote{The paper \cite{B2}  cites \cite{GOE} for the base case of GOE, but as explained \cite{GOE} only proves the convergence for the first marginal of $\Upsilon^{(1)}_n$.} 

	As a direct corollary of Theorem \ref{small gaps of GbetaE}, we obtain the limiting distribution of the $k$-th smallest gap in the bulk of the semicircle law. %Its proof is deferred to the end of Section \ref{sec:pfoutline}.
	\begin{cor}
		\label{density of small gap}
		Let $I$ be any interval in $(-2+\ep,2-\ep)$, and  $t_{k}^{(\beta)}(I)$ be the $k$-th smallest  gap for eigenvalues falling in $I$, % i.e.,  the $k$-th smallest number of the set $
		% 	\left\{\lambda_{i+1}-\lambda_{i}| \lambda_i \in I\right\}.$
		we define the 
		normalized gap $${\tau}_{k}^{(\beta)}(I):=
		\left(\dfrac{1}{(\beta+1)c_{\beta}}\int_{I}\left(2\pi\rho_{\semi}(x)\right)^{\beta+2}\mathrm{d}x\right)^{\frac{1}{\beta+1}}
		n^{\frac{\beta+2}{\beta+1}}t_{k}^{(\beta)}(I).$$
		Then for $\beta = 1,2$ and $4$, we have 
		$$
		\label{pdf of small gaps}
		\lim_{n\rightarrow \infty}\mathbb{P}\left({\tau}_{k}^{(\beta)}(I)\in A\right) =\int_{A}\frac{\beta+1}{(k-1) !} x^{k(\beta+1)-1} e^{-x^{\beta+1}}\mathrm{d}x
		$$
		for any bounded  interval $A\subset (0,+\infty)$.
	\end{cor}
	%Again, the $\beta=1$ case of Corollary \ref{density of small gap} coincides with  \cite[Corollary 1]{GOE}, as expected. 
	%In fact, Corollary \ref{density of small gap} also agrees with the small gap distribution for Circular-$\beta$ Ensemble (with $\beta$ taking to be 1,2,4). This can be explained  from the local sine-kernel universality of G$\beta$E and C$\beta$E.

	% \subsection{Outline of the proof}\label{subsec:stra}
	%A general intuition for the extreme gap problem is the so-called \emph{Poisson ansatz}: treat the eigenvalue gaps as if they were i.i.d. sampled from Gaudin-Mehta distributions.  Indeed, it is generally expected that eigenvalues  should behave almost independently
	%when they are distant from each other. %(Here `distant' means the distances of the eigenvalues are much larger than the mean spacings).
	%In connections with the small gap problem, 
	%we note that several recent works \cite{CEX, H} consider the probability of simultaneous existences of eigenvalues/singular values in multiple distant intervals. See \cite[Proposition 3.8]{CEX} for  non-Hermittian matrices 
	%and \cite[Theorem 1.5]{H} for real symmetric matrices. 
	
	Now we outline the main steps to prove Theorem \ref{small gaps of GbetaE}. 
	%We will adapt the strategy originally developed by Soshnikov \cite{So}, along with some modifications, together with the Pfaffian structure of GSE. This strategy is also used in \cite{BB, FTW, FW2}.  
	Given the eigenvalues $\lambda_1<\lambda_2< \cdots<\lambda_n$ of GSE. For given $\epsilon>0$ small,
	consider the point processes consisting of the eigenvalues of GSE in the bulk $(-2+\epsilon,2-\epsilon)$, 
	$$ 
	\label{process of eigenvalues}
	\xi_{n}:=\sum_{i=1}^n \delta_{\lambda_i}\mathbbm{1}[\abs{\lambda_i}<2-\epsilon]. 
	$$
	Now we construct an auxiliary  point process $\tilde{\xi}_n\subset \xi_n $ as follows. 
	For any bounded measurable set $A\subset (0,+\infty)$,  let $A_n = n^{-6/5} A$. The process $\tilde\xi_n$ is selected from  $\xi_n$ such that the following two conditions hold:
	\begin{itemize}
		\item[-] $|\xi_{n}\left(\lambda_k+A_n\right)|=1.$ 
		\item[-] There does not exists $\lambda_{\ell} \neq \lambda_k$ such that $|\xi_{n}\left(\lambda_{\ell}+A_n\right)|=1$
		and $\abs{\lambda_k-\lambda_{\ell}}<\log^{-1} n$.
	\end{itemize} 
	Theorem \ref{small gaps of GbetaE} is a direct consequence of the following two lemmas. We first have 
	%Lemma \ref{lemt} and Lemma \ref{lemf}. 
	\begin{lem} \label{lemt}The two point processes $\Upsilon_{n}^{(4)}(A\times \cdot)$ and $\tilde{\xi}_{n}(\cdot)$ are asymptotically equivalent, i.e., for any bounded interval $I\subset (-2+\ep,2-\ep)$, the cardinalities satisfy  
		$$\abs{\Upsilon_{n}(A\times I)}-
		|{\Tilde{\xi}_{n}(I) }| \to 0 \quad \mbox{in probability as $n\to\infty$}. $$
	\end{lem}
	Next, we prove that  $\tilde{\xi}_n$ converges to a Poisson point process by the moment method. %Let $\Tilde{\rho}_{k}$ be the correlation function of $\tilde{\xi}_n$, then we have		
	\begin{lem} \label{lemf}For any fixed positive integer $k$, one has the following convergence of the $k$-th factorial moment
		\begin{align}\label{factorial_conv}
			\begin{aligned}			\lim_{n\to\infty}\mathbb{E}\left(\frac{|\tilde{\xi}_{n}(I)| !}{\left(|\tilde{\xi}_n(I)|-k\right) !}\right) =\left(\dfrac{1}{540\pi^2}\int_A u^{4}\mathrm{d}u\right)^k\left(\int_{I^{k}} \left( 2\pi\rho_{\semi}(\lambda)\right)^{6} \mathrm{d} \lambda\right)^k.
			\end{aligned}
		\end{align}
	\end{lem}

	We now provide a brief overview of the key steps of the proof. The starting point is the Pfaffian structure of GSE, where the correlation function is the Pfaffian of an anti-symmetric matrix (see \eqref{anticor} and \eqref{kn4exp}). This matrix can be expressed in terms of a scalar kernel $S_n$, along with its derivative and antiderivative (see \eqref{def of kernel 4}). 
	
	There are several significant challenges in analyzing the correlation functions of Pfaffian processes compared to determinantal processes. For instance, because the kernel matrices for Pfaffian processes are anti-symmetric, the Hadamard-Fischer inequality for positive symmetric matrices, as used in \cite{BB, So}, can no longer be applied to estimate their correlation functions. Additionally, the growth orders of different entries in the anti-symmetric matrix are mixed and interwoven (see Lemma \ref{order of multi kernel} below).
	
	In Lemma \ref{lem:diag pfa}, our first observation is that the integral of the product of the Pfaffians of the on-diagonal $4\times 4$ blocks provides the leading order term for the factorial moments. Therefore, one of the main tasks is to demonstrate that the contributions from the off-diagonal $4\times 4$ blocks are negligible in the limit. %For determinantal processes, the off-diagonal entries or blocks are usually negligible due to the rapid decay of kernels, and thus the Hadamard-Fischer inequality for positive symmetric matrices can be applied to estimate correlation functions. 
	To address this, in Section \ref{subsub:main}, we will first perform two rounds of transformations on the anti-symmetric kernel matrix with its Pfaffian unchanged, so that the  orders of different entries become more balanced (e.g., Lemma \ref{lem:round1}).  Then, the negligibility of the contribution from  off-diagonal blocks (e.g., Lemma \ref{lem:bdoffdiag}) follows from the combinatorial counting arguments (e.g., Lemma \ref{lem:ord1bd}). In Section \ref{sec:err}, the similar arguments are applied to prove the negligibility of two error terms ($E_{1,k,n}$ and $E_{2,k,n}$ in \eqref{lower}) appearing when estimating the correlation functions. This approach  provides an alternative proof for  GOE, which is sketched in Section \ref{sec:pfgoe}.  We anticipate that these computations, particularly the two rounds of transformations and combinatorial counting arguments, can be applied to study the intricate structures of other Pfaffian point processes.%having a similar structures to GOE and GSE, where the kernel matrix  is given  by a scalar kernel along with its derivative and antiderivative. 
	
	%To the best of our knowledge, this is the first treatment on the Pfaffian of the anti-symmetric kernel matrix to study the local statistics of a Pfaffian point process beyond the average scale. 
	%	The paper is organized as follows. In Section \ref{sec:preli}, we will provide some preliminary knowledge, e.g., correlation functions and Hermite polynomials. In Section \ref{sec:kernel}, we will derive several uniform estimates regarding the kernel of GSE.   In Section \ref{sec:pfoutline}, we outline the main steps to prove Lemma \ref{lemt} and Lemma \ref{lemf}, and the detailed proofs are presented in Section \ref{subsec:main} and Section \ref{sec:err}.  In  Section \ref{sec:pfgoe},  we sketch the proof for GOE. 
	
	\bigskip 
	
	\emph{Notation.} In this paper,  we use $C$ to denote some uniform constant whose specific value may vary from line to line. For two sequences of  real numbers  $x_n$ and $y_n$, we write $x_n \lesssim y_n$  or $x_n=O(y_n)$ if $\abs{x_n}\leq C \abs{y_n}$ for some $C>0$. We write $x_n \sim y_n$ if $\lim_{n\to\infty}x_n/y_n=1$. Given a matrix $A$, we denote $A(i,j)$ as its $(i, j)$-entry.  For two sequences of matrices $M_n$ and $M_n'$ of the same (fixed) size, we write $M_n \lesssim  M'$ if $\abs{M_n(i,j)}\lesssim  \abs{M'_n(i,j)}$ for every $(i,j)$-entry.  Given a set $S$, $|S|$ denotes its cardinality. 
	
	\section{Preliminaries}\label{sec:preli}
	\subsection{Correlation functions}
	First, we review some basic concepts
	regarding the factorial moments and the correlation functions of a point process.
	Let 
	\begin{align*}
		\xi = \sum_{i}\delta_{X_i}
	\end{align*}
	be a simple point process on $\mathbb{R}$. For any $k\geq 1$, we can construct a new point process $\xi_k$ in $\mathbb{R}^k$ via
	$$
	\xi_{k} = \sum_{X_{i_1},\ldots,X_{i_k} \textrm{ pairwise distinct}}\delta_{(X_{i_1},\cdots, X_{i_k})}.
	$$
	The $k$-{point correlation function}  of $\xi$
	is a function $\rho_k$ on $\mathbb{R}^k$ such that for any bounded Borel sets $B_1,\ldots, B_k$, it holds that
	$$
	\E|\xi_k\left( B_1\times \cdots\times B_k\right)|=\int_{B_1\times \cdots\times B_k}\rho_k(x_1,...,x_k)\mathrm{d}x_1\cdots\mathrm{d}x_k.
	\label{definition of correlation function}
	$$
	In particular, for any bounded Borel set $B\subset \mathbb{R}$, one has 
	$$
	\mathbb{E}\left(\frac{|\xi(B)| !}{(|\xi(B)|-k) !}\right) =\int_{B^k}\rho_k(x_1,...,x_k)\mathrm{d}x_1\cdots\mathrm{d}x_k, 
	\label{property of product measure}
	$$
	which is  the $k$-th {factorial moment} of the cardinality $|\xi(B)|$.

	\subsection{Pfaffian}
	Now we recall the definition of the Pfaffian of an anti-symmetric matrix. Let  $\M=\left(\M(i,j)\right)_{1\leq i, j\leq 2 N}$  be a $2N\times 2N$ anti-symmetric matrix, where $N$ is a positive integer. Then the Pfaffian of $\M$ is defined by 
	$$
	\pfa \M := \sum_{\sigma\in \mathcal{P}_{2N}}\mathrm{sgn}(\sigma)\M(\sigma_1,\sigma_2)\M(\sigma_3,\sigma_4)\cdots \M(\sigma_{2N-1},\sigma_{2N}),
	\label{definition of pfaffian}
	$$
	where $\mathcal{P}_{2N}$  consists of all permutations $\sigma$  of $\{1,2,...,2N\}$ such that   
	$$
	\sigma_1<\sigma_3<\cdots<\sigma_{2N-1}, 
	\text{ and } \sigma_1<\sigma_2,\sigma_3<\sigma_4,...,\sigma_{2N-1}<\sigma_{2N}.
	\label{definition of P}
	$$
	For any $2N\times 2N$ matrix $A$, we have
	$$\label{keypfaid}
	% (\pfa X)^{2}=\det X, 
	\pfa\left(A^{T} \M A\right)=(\det A)(\pfa \M). %\pfa(\lambda X)=\lambda^{N} \pfa X .
	$$
	This identity implies that  $\pfa \M$ is invariant under congruent transformations when $A$ has unit determinant, which in particular includes the case of the
	product of elementary matrices. %We will frequently use this fact later.   
	%Here, these identities follow from the definition of Pfaffian.
	
	\subsection{Hermite polynomials}\label{sec:hermite}
	The monic orthogonal polynomials on $\mathbb{R}$ for the weight
	$$
	\mathrm{d} v(x)=w(x) \mathrm{d} x, \quad w(x)=\frac{e^{-x^{2} / 2}}{\sqrt{2 \pi}}
	$$%$\{H_{k}(x), k \geq 0\} $, 
	are called {Hermite polynomials},  which are given by
	$$
	H_{n}(x):=(-1)^{n} e^{x^{2} / 2} 
	\frac{\mathrm{d}^n }{\mathrm{d}x^n}
	\left(e^{-x^{2} / 2}\right)
	$$ for integer $n\geq0$.
	The $L^2$-normalized {Hermite wave functions}
	$$
	\label{def of wave}
	\varphi_{n}(x):=\frac{e^{-x^{2} / 4} H_{n}(x)}{(2 \pi)^{1 / 4} \sqrt{n !}}
	$$
	are orthonormal, i.e., 
	\begin{align*}
		\int_{\mathbb{R}}\varphi_{n}(x)\varphi_{m}(x)\mathrm{d}x = \delta_{n,m}.
	\end{align*}
	For any $\theta\in(0,\pi)$, we set
	
	%a(\theta) &= \sin 2\theta - 2\theta,\\
	$$b(\theta)=\theta-\frac{1}{2} \sin 2 \theta, \quad R_{n}(\theta)=\sin \left(n b(\theta)+\frac{\pi}{4}-\frac{\theta}{2}\right),$$  
	$$Q_{n}(\theta)=\frac{3}{16}(\sin \theta)^{-2} \sin \left(n b(\theta)-\frac{3}{4} \pi-\frac{5}{2} \theta\right)+\frac{5}{96}(\sin \theta)^{-3} \sin \left(n b(\theta)-\frac{5}{4} \pi-\frac{7}{2} \theta\right).$$
	Then we have the following classical Plancherel-Rotach formula (see \cite{DY}), 
	%\label{Plancherel-Rotach asymptotics}
	% 		\begin{align}
	%			e^{-x^{2}/4} H_{n}\left(2\sqrt{n+1}\cos \theta\right)
	%			=(\sin \theta)^{-{1}/{2}}(n !)^{{1}/{2}} 2^{{1}/{4}}(\pi n)^{-{1}/{4}}\left[R_{n+1}(\theta)+\frac{1}{n} Q_{n}(\theta)+O\left(n^{-2}\right)\right],
	%		\end{align}
	%		or, equivalently,
	\begin{equation}\label{Plancherel-Rotach Formula}
		\begin{split}	&\varphi_n\left(2\sqrt{n+1}\cos \theta\right) \\ &=(\pi\sin \theta)^{-{1}/{2}}n^{-{1}/{4}} 
			\left(R_{n+1}(\theta)+\frac{1}{n} \left(Q_{n+1}(\theta)-\frac{5}{24} R_{n+1}(\theta)\right)+O\left(n^{-2}\right)\right).\end{split}
	\end{equation}
	Here,  for any $\epsilon>0$, the error bound $O\left(n^{-2}\right)$  is uniform for $\theta\in(-\pi+\epsilon,\pi-\epsilon)$. The Plancherel-Rotach formula will imply the following % which will  be used in estimating the   kernel of GSE later on.
	%We defer its proof to Appendix \ref{app:hermite}.
	\begin{lem}\label{hermitintbd}
		Given any $\ep>0$, for $x\leq y$ such that $x,y\in (-2+\epsilon,2-\epsilon)$, one has uniform estimates, 
		\begin{equation}\label{phin4}
			\int_{\sqrt{n}x }^{\sqrt{n}y} \varphi_{n}(t)\mathrm{d}t=O\left(n^{-3/4}\right)
		\end{equation}
		and \begin{align}
			\int_{\sqrt{2n}x}^{+\infty}\varphi_{2n-1}(t)\mathrm{d}t = O\left(n^{-1/4}\right).
			\label{int 3 of hermite}
		\end{align}
	\end{lem}
	
	\begin{proof} 
		%We first consider \eqref{phin4}. %We only prove the first equation in \eqref{phin4} since the second one is similar. 
		Let $0<\theta_1<\theta_2<\pi$ be such that
		$
		2\sqrt{n+1}\cos\theta_1=\sqrt{n}y,\,\, 
		2\sqrt{n+1}\cos\theta_2=\sqrt{n}x.
		$
		Since $(x,y)\in (-2+\ep,2-\ep)$, $\theta_1$ and $\theta_2$ are also bounded away from 0 and $\pi$ for large $n$. 
		By a change of variable $t=2\sqrt{n+1}\cos \theta$ and
		using \eqref{Plancherel-Rotach Formula}, we have
		\begin{equation}\label{phin1}
			\begin{split}
				\int_{\sqrt{n}x}^{\sqrt{n}y} \varphi_{n}(t)\mathrm{d}t&=2\sqrt{n+1}
				\int_{\theta_1}^{\theta_2} \varphi_{n}\left(2\sqrt{n+1}\cos \theta\right)\sin\theta \mathrm{d}\theta\\
				&=2\sqrt{n+1}
				\pi^{-1/2} n^{-1/4}
				\int_{\theta_1}^{\theta_2} R_{n+1}(\theta)(\sin\theta)^{1/2}\mathrm{d}\theta +O\left(n^{-3/4}\right).
			\end{split}
		\end{equation}
		The integrand $R_{n+1}(\theta)(\sin\theta)^{1/2}$ can be written as $g_1(\theta)\sin \left((n+1)b(\theta)\right)+g_2(\theta)\cos((n+1)b(\theta))$ for some $C^1(0,\pi)$ functions $g_1$ and $g_2$, where $b(\theta)=\theta-\sin(2\theta)/2$ belongs to $C^1(0,\pi)$ and satisfies
		$
		b'(\theta)=1-\cos (2\theta)>c>0$ on $(\theta_1,\theta_2).
		$
		Integration by parts will simply imply	\begin{equation}\label{phin2}
			\abs{  \int_{\theta_1}^{\theta_2} R_{n+1}(\theta)(\sin\theta)^{1/2}\mathrm{d}\theta }\lesssim n^{-1}.
		\end{equation}
		This will prove  \eqref{phin4} by  \eqref{phin1} and \eqref{phin2}. 
		%	we get
		%	\begin{equation}\label{phin3}
		%		\abs{    \int_{\sqrt{n}x}^{\sqrt{n}y} \varphi_{n}(t)\mathrm{d}t}\lesssim n^{-3/4}.
		%	\end{equation}
		By identities (7.376)  in  \cite{GR}, we can derive 
		\begin{align}
			\int_{0}^{+\infty}e^{-x^2/4} H_{2n}(x) \mathrm{d}x &= (2n-1)!!\sqrt{\pi},\quad	\int_{0}^{+\infty}e^{-x^2/4} H_{2n+1}(x) \mathrm{d}x \lesssim (2n)!!\sqrt{\pi}.
			\label{int 2 of hermite}
		\end{align} By Stirling's formula,  the integral of the Hermite wave functions satisfies 
		\begin{equation}\label{phin5}
			\begin{split}
				\int_{0}^{+\infty}\varphi_{2n-1}(t) \mathrm{d}t \lesssim &  \dfrac{(2n-2)!!}{\sqrt{(2n-1)!}} 
				\lesssim  \dfrac{\sqrt{4(n-1)}\left(\frac{2n-2}{e}\right)^{n-1}}{\left(2\pi(2n-1)\right)^{1/4}\left(\frac{2n-1}{e}\right)^{n-1/2}} \lesssim
				n^{-1/4}.
			\end{split} 
		\end{equation} By the decomposition 
		$\int_{\sqrt{2n}x}^{+\infty}\varphi_{2n-1}(t)\mathrm{d}t=\int_{0}^{+\infty}\cdots-\int_{0}^{\sqrt{2n}x}\cdots$, 
		\eqref{int 3 of hermite} follows directly from  \eqref{phin4} and \eqref{phin5}.		
	\end{proof}

	\section{Kernel estimates of  GSE}\label{sec:kernel}
	GSE is a Gaussian measure on the space of Hermitian self-dual matrices \cite{M}. The eigenvalues of GSE are real, and the joint density is given by \eqref{jpdf of Gauss ensemble} with $\beta=4$. The eigenvalues of GSE form a Pfaffian point process. Specifically,  for any integer $k\geq 1$, the $k$-point correlation function   can be expressed as the Pfaffian of a $2k\times 2k$ anti-symmetric matrix  (see Section 3.9 in \cite{AGZ})	\begin{equation}\label{anticor}
		\rho_{k}(\lambda_1,...,\lambda_k) = \pfa\left(JK_n(\lambda_i,\lambda_j)\right)_{1\leq i,j\leq k},
	\end{equation}
	where 
	\begin{equation*}
		J = \begin{bmatrix}
			0  & 1 \\
			-1 & 0
		\end{bmatrix},
	\end{equation*}
	and $K_n(x,y)$ is  the kernel given by
	\begin{equation}\label{kn4exp}
		K_n(x,y) =
		\begin{bmatrix}
			S_n(x,y)  & 
			V_n(x,y) \\ J_{n}(x, y) & S_n(y,x)
		\end{bmatrix}.
	\end{equation}
	Here, 
	\begin{align} 
		\begin{aligned}
			S_n(x, y)
			&=\frac{\sqrt{2n}}{2}
			\sum_{j=0}^{n-1} \varphi_{j}\left(\sqrt{2n}x\right) \varphi_{j}\left(\sqrt{2n}y\right)-\frac{n}{2}\varphi_{2n}(\sqrt{2n}x)\int_{\sqrt{2n}y}^{\infty} \varphi_{2n-1}(t)\mathrm{d}t,\\
			V_{n}(x, y) & =-{\partial_y} S_{n}^{}(x, y), \\
			J_{n}(x, y) & = \int_y^{x} S_n(t,y)\mathrm{d}t.
		\end{aligned}
		\label{def of kernel 4}
	\end{align}
	The $2\times 2$ matrix $JK_n(x,y)$ is anti-symmetric where
	\begin{equation}
		\label{relation of 2 S(x,y)}\begin{aligned}
			V_n(x,y) = - V_n(y,x), \quad 
			\int_{y}^{x}S_n(t,y)\mathrm{d}t = \int_{y}^{x}S_n(t,x)\mathrm{d}t. \end{aligned}
	\end{equation}
	%	The proofs of 	\eqref{relation of 2 S(x,y)} and \eqref{relation of 2 integral} follow from recurrence relations of Hermite polynomials. Alternatively, they can be seen from  	%\subsection{Convergence of the rescaled kernel}
	For any given $x_0 \in (-2,2)$, we define the following rescaled GSE kernel 
	\begin{equation}\label{rescaledkernel}
		\begin{aligned}	&\widehat{K}_n(x, y):=\\& \frac{1}{n\rho_{\semi}(x_0)}\begin{bmatrix}
				\frac{1}{\sqrt{n \rho_{\semi}(x_0)}} & 0 \\
				& \sqrt{n \rho_{\semi}(x_0)}
			\end{bmatrix} K_{n}(x, y)\begin{bmatrix}
				\sqrt{n \rho_{\semi}(x_0)} & 0 \\
				0 & \frac{1}{\sqrt{n \rho_{\semi}(x_0)}}
			\end{bmatrix} . \end{aligned}
	\end{equation}
	We further denote it by 
	\begin{align*}
		\widehat{K}_n(x, y):=	\begin{bmatrix}
			\widehat{S}_{n}(x, y) & \widehat{V}_n(x, y)\\
			\widehat{J}_n(x, y) & \widehat{S}_n(y,x)
		\end{bmatrix}.
	\end{align*}
	Let $K_{\sin}$   be  the {sine kernel}
	\begin{equation}\label{sinekerneldef}
		K_{\sin}(t) := \frac{\sin(\pi t)}{\pi t}. 
	\end{equation}
	We have the following uniform estimates regarding the rescaled GSE kernel. 
	\begin{lem}
		\label{Convergence of the rescaling kernel}
		For any $\ep>0$, uniformly over $x_0 \in (-2+\ep,2-\ep)$ and $x,y\in [-1,1]$,    
		\begin{align}
			\label{convergence of S}
			&\left\vert  \widehat{S}_{n}\left(x_0+\frac{x}{n \rho_{\semi}(x_0)}, x_0+\frac{y}{n \rho_{\semi}(x_0)}\right)-K_{\sin}(2(x-y)) \right\vert\lesssim\frac{1}{\sqrt{n}}, \\
			\label{convergence of I}
			&\left|\widehat{J}_{n}\left(x_0+\frac{x}{n \rho_{\semi}(x_0)}, x_0+\frac{y}{n \rho_{\semi}(x_0)}\right)-\int_{0}^{x-y} K_{\sin}(2 t) \mathrm{d} t\right|\lesssim\frac{\abs{x-y}}{\sqrt{n}},\\
			\label{convergence of D}
			& \left| \widehat{V}_{n}\left(x_0+\frac{x}{n \rho_{\semi}(x_0)}, x_0+\frac{y}{n \rho_{\semi}(x_0)}\right)-{\partial_x} [K_{\sin}(2(x-y))] \right|\lesssim\frac{1}{n}.
		\end{align}

	\end{lem}
	
	\begin{proof}
		Let $K_n^{(2)}(x,y)$ be the kernel of the GUE. It can be expressed in terms of Hermite wave functions as follows (e.g., Lemma 3.2.2 in \cite{AGZ}):
		$$
		K_n^{(2)}(x,y)=\sqrt{n}\sum_{i=0}^{n-1}\varphi_i(\sqrt{n}x)\varphi_i(\sqrt{n}y).
		$$
		One first has the uniform estimates (e.g., Theorem 1 in   \cite{KS})
		%	\begin{align}
		%		&\sup _{x, y \in\left[-1, 1\right]}\left|\frac{1}{n \rho_{\semi}(x_0)} K_{n}^{(2)}\left(x_0+\frac{x}{n \rho_{\semi}(x_0)}, x_0+\frac{y}{n \rho_{\semi}(x_0)}\right)-K_{\sin}(x-y)\right|=O\left(\frac{1}{n}\right), \label{con of K2}\\
		%		&\sup _{x, y \in\left[-1,1\right]}\left|{\partial_x}\left(\frac{1}{n \rho_{\semi}(x_0)} K_{n}^{(2)}\left(x_0+\frac{x}{n \rho_{\semi}(x_0)}, x_0+\frac{y}{n \rho_{\semi}(x_0)}\right)-K_{\sin}(x-y)\right)\right|=O\left(\frac{1}{n}\right).
		%		\label{con of dev of K2}
		%	\end{align}
		
		\begin{align}
			&\sup _{x, y \in\left[-1, 1\right]}
			\left|\frac{1}{n \rho_{\semi}(x_0)} 
			K_{n}^{(2)}\left(x_0+\frac{x}{n \rho_{\semi}(x_0)}, x_0+\frac{y}{n \rho_{\semi}(x_0)}\right)
			-K_{\sin}(x-y)\right| \notag \\
			&\lesssim \frac{1}{n}, \label{con of K2}\\
			&\sup _{x, y \in\left[-1,1\right]}
			\left|{\partial_x}\left(\frac{1}{n \rho_{\semi}(x_0)} 
			K_{n}^{(2)}\left(x_0+\frac{x}{n \rho_{\semi}(x_0)}, x_0+\frac{y}{n \rho_{\semi}(x_0)}\right)
			-K_{\sin}(x-y)\right)\right| \notag \\
			&\lesssim \frac{1}{n}.
			\label{con of dev of K2}
		\end{align}
		By \eqref{def of kernel 4}, we have the relation \begin{align}
			\label{relationship of k2 and k4}
			S_n(x, y)=\frac{K_{2 n}^{(2)}(x, y)}{2}-\frac{n}{2}\varphi_{2n}(\sqrt{2n}x)\int_{\sqrt{2n}y}^{\infty} \varphi_{2n-1}(t)\mathrm{d}t.
		\end{align}
		The estimate \eqref{convergence of S} then follows from the bound $O\left(n^{-1/4}\right)$ for $\varphi_{2n}$  (see \eqref{Plancherel-Rotach Formula}) and  the bound $n^{-1/4}$ for the integral term (see \eqref{int 3 of hermite}). 
		We also obtain \eqref{convergence of I} by integrating both sides of \eqref{convergence of S}.  	
		Taking the derivative with respect to $y$ on \eqref{relationship of k2 and k4}, by \eqref{Plancherel-Rotach Formula} again, we get
		\begin{align}
			\abs{{\partial_y}S_n(x,y)-\frac{1}{2}{\partial_y}K^{(2)}_{2n}(x,y)} = \abs{\frac{\sqrt{2}n^{3/2}}{2}\varphi_{2n}(\sqrt{2n}x)\varphi_{2n-1}(\sqrt{2n}y)} \lesssim n.
			\label{snskdediff}
		\end{align}
		Then, \eqref{convergence of D} follows from the estimate \eqref{con of dev of K2}.
		
	\end{proof}%Now we derive some uniform upper bounds for GSE kernel.  
	%By Lemma \ref{Convergence of the rescaling kernel}, 
	%for $x_0\in(-2,2)$ and $x,y=O\left(n^{-1/5}\right)$, we have the uniform estimate 
	%$$\begin{aligned}
	%	\label{estimate of the kernel}
	%\abs{ 		\widehat{K}_{n}\left(x_0+\frac{x}{n \rho_{\semi}(x_0)},x_0+\frac{y}{n\rho_{\semi}(x_0)}\right)-K(x, y) }\lesssim \left(\begin{array}{ll}
	%			n^{-1/2}  & n^{-1} \\
	%			n^{-7/10}   & n^{-1/2}
	%		\end{array}\right),
	%	\end{aligned}$$
	%	 where $K$ is the limiting kernel 
	%	$$\begin{aligned}
	%		K(x,y) :=& 
	%		\begin{bmatrix}
	%			K_{\sin}(2x-2y) & {\partial_x}K_{\sin}(2x-2y) \\
	%			\int_{0}^{x-y}K_{\sin}(2t)\mathrm{d}t & K_{\sin}(2x-2y)
	%	\end{bmatrix}.\label{K4}
	%\end{aligned}$$
	Subsequently, we need to control the behavior of the kernel at different scales. To achieve this, we provide the following uniform upper bounds. 	
	\begin{lem}
		\label{order of multi kernel}
		Let $d:=d(x,y)=\abs{x-y}$.  For any $\ep>0$ and $(x,y)\in (-2+\ep,2-\ep)$,
		\begin{align}
			\abs{S_n(x,y)} \lesssim\min\left\{\frac{1}{d},n\right\}+n^{1/2},\label{s4rough}
		\end{align}
		
		\begin{equation}\label{s4derough}
			\abs{ \partial_x S_n(x,y)}\lesssim 
			n\min\left\{\frac{1}{d},n^2d\right\}+n^{3/2},
		\end{equation}
		\begin{equation}\label{s4derough2}
			\abs{V_{n}(x,y)}=\abs{\partial_y S_n(x,y)}\lesssim n\min\left\{\frac{1}{d},n^2d\right\}+n,
		\end{equation}
		\begin{equation}\label{fnde}
			\abs{\partial_x V_{n}(x,y)}+
			\abs{\partial_y V_{n}(x,y)}\lesssim n^2\min\left\{n,\frac{1}{d}\right\},
		\end{equation} 
		\begin{equation}\label{fn2de}
			\abs{\partial^2_{xy} V_{n}(x,y)}+
			\abs{\partial^2_{yx} V_{n}(x,y)}+\abs{\partial^2_{yy} V_{n}(x,y)}\lesssim n^3\min\left\{n,\frac{1}{d}\right\},
		\end{equation} 
		\begin{equation}\label{intbd}
			\abs{\int_x^y S_n(t,w)\mathrm{d}t  }\lesssim \log n,\quad
			\forall w\in (-2+\ep,2-\ep).
		\end{equation}
	\end{lem}
	
	\begin{proof}
		We first prove the following estimates  for the GUE kernel $K_n^{(2)}$: 
		\begin{equation}\label{k2bd}
			\begin{aligned}
				\abs{K^{(2)}_n(x,y)}\lesssim& \min\left\{\frac{1}{d},n\right\}, \\
				\abs{ \partial_x K^{(2)}_n(x,y)}+\abs{\partial_yK^{(2)}_n(x,y)}\lesssim& n\min\left\{\frac{1}{d},n^2d \right\}+n.
			\end{aligned}
		\end{equation}
		When $d\leq 1/n$, \eqref{k2bd} directly follows from \eqref{con of K2} and \eqref{con of dev of K2}. 
		For $d\geq 1/n$,
		by Christoffel-Darboux formula, we can rewrite (e.g., Lemma 3.2.5 in \cite{AGZ})
		\begin{equation}\label{cdformula}
			K_n^{(2)}(x,y)=\sqrt{n}\frac{\varphi_n(\sqrt{n}x)\varphi_{n-1}(\sqrt{n}y) -
				\varphi_n(\sqrt{n}y)\varphi_{n-1}(\sqrt{n}x)
			}{x-y}  .
		\end{equation}
		By Plancherel-Rotach formula \eqref{Plancherel-Rotach Formula}, we have
		%\begin{equation}\label{phin-1n}
		$	\abs{\varphi_n(\sqrt{n}x)}+\abs{\varphi_{n-1}(\sqrt{n}x)} \lesssim n^{-1/4}. 
		$
		This proves the first inequality in \eqref{k2bd}. We note the relation  (e.g., Lemma 3.2.7 in \cite{AGZ})
		$$
		\frac{\mathrm{d}}{\mathrm{d}x}\left( \varphi_n(\sqrt{n}x)\right)=n\varphi_{n-1}(\sqrt{n}x)-\frac{nx}{2}\varphi_n(\sqrt{n}x).
		$$
		The second  inequality in \eqref{k2bd} follows by differentiating \eqref{cdformula} with respect to $x$ together with the estimate  \begin{equation}\label{dephibd}
			\begin{split}
				\abs{  
					\frac{\textrm{d}}{\textrm{d} x} \left(\varphi_n(\sqrt{n}x)
					\right)}
				\leq  		n\left(\abs{\varphi_{n-1}(\sqrt{n}x)}+\abs{\varphi_n(\sqrt{n}x)}\right)\lesssim n^{3/4}.
			\end{split}
		\end{equation}
		%We have used \eqref{phin-1n} in the last inequality of \eqref{dephibd}. 
		%This completes the proof of  \eqref{k2bd}. 
		In fact, performing further differentiations, we have
		\begin{equation}\label{kndehi}
			\begin{split}
				\abs{    \frac{\mathrm{d}^2}{\mathrm{d}x^2} \left(\varphi_n(\sqrt{n}x) \right)}&\lesssim n^{7/4},\\ 
				\sum_{i+j=2, i,j\geq 0} \abs{\partial^i_x\partial^j_y K_n^{(2)}(x,y)}&\lesssim n^2\min\left\{\frac{1}{d},n\right\},\\
				\sum_{i+j=3, i,j\geq 0} \abs{\partial^i_x\partial^j_y K_n^{(2)}(x,y)}&\lesssim n^3\min\left\{\frac{1}{d},n\right\}.
			\end{split}
		\end{equation}
		Now we are ready to prove  \eqref{s4rough} - \eqref{intbd}.   
		To prove \eqref{s4rough}, by \eqref{Plancherel-Rotach Formula}, \eqref{phin4}, \eqref{int 3 of hermite},  \eqref{relationship of k2 and k4} and \eqref{k2bd}, we get
		\begin{equation}
			\begin{split}
				\abs{S_n(x,y)}\leq& \frac{1}{2}\abs{K_{2n}^{(2)}(x,y)}+\frac{n}{2} \abs{\varphi_{2n}\left(\sqrt{2n}x\right)}\abs{\int_{\sqrt{2n}y}^{\infty}  \varphi_{2n-1}(t) \mathrm{d} t} \\
				\lesssim& \min\left\{\frac{1}{d},n\right\}+n^{1/2}.
			\end{split}
		\end{equation}
		To prove \eqref{s4derough}, by \eqref{int 3 of hermite}, \eqref{k2bd} and \eqref{dephibd}, we get
		\begin{equation}
			\begin{split}
				\abs{ \partial_xS_n(x,y)}&\lesssim
				\frac{1}{2}\abs{\partial_x K_{2n}^{(2)}(x,y)}+\frac{n}{2}\abs{\frac{\mathrm{d}}{\mathrm{d}x}
					\left(   \varphi_{2n}(\sqrt{2n}x) \right)} \cdot \abs{\int_{\sqrt{2n}y}^{\infty}  \varphi_{2n-1}(t) \mathrm{d} t}\\
				&\lesssim  n\min\left\{\frac{1}{d},n^2d \right\}+n^{3/2}.
			\end{split}
		\end{equation}
		Now,  \eqref{s4derough2} follows directly from \eqref{snskdediff} and the second part of \eqref{k2bd}. And   \eqref{fnde} and \eqref{fn2de}
		can be proved similarly, thanks to \eqref{kndehi}. It remains to show \eqref{intbd}. We have 
		\begin{equation}\label{intbd2}
			\begin{aligned}
				\abs{\int_x^y S_n(t,w)\mathrm{d}t  }\leq&
				\frac{1}{2} 				\abs{\int_x^y K_{2n}^{(2)}(t,w)\mathrm{d}t}\\
				&+\frac{n}{2}
				\abs{\int_x^y \varphi_{2n}(\sqrt{2n}t)\mathrm{d}t} \cdot \abs{\int_{\sqrt{2n}w}^{\infty}  \varphi_{2n-1}(t) \mathrm{d} t}.
			\end{aligned}  
		\end{equation}
		By \eqref{k2bd}, we can bound the first term in \eqref{intbd2} as follows:
		\begin{equation}\label{intbd3}
			\int_x^y \abs{K_{n}^{(2)}(t,x)}\mathrm{d}t
			\lesssim\left(\int_0^{1/n} n\mathrm{d}t+\int_{1/n}^4 \frac{1}{t}\mathrm{d}t \right)\lesssim\log n.
		\end{equation}
		By \eqref{phin4} and \eqref{int 3 of hermite}, we can bound the second term in \eqref{intbd2} by
		$
		Cn\cdot n^{-5/4}\cdot  n^{-1/4}= Cn^{-1/2},
		$  which completes the proof of
		\eqref{intbd}.  
		
	\end{proof}

	\section{Main lemmas }\label{sec:pfoutline}
	
	%In this section we will give the proof outline of Theorem \ref{small gaps of GbetaE} for the GSE case where $\beta=4$.    For a given $\epsilon>0$,
	%consider the point processes consist of the eigenvalues of GSE in the bulk $(-2+\epsilon,2-\epsilon)$, 
	%	\begin{align}
	%		\label{process of eigenvalues}
	%	\xi_{n}:=\sum_{i=1}^n \delta_{\lambda_i}\mathbbm{1}[\abs{\lambda_i}<2-\epsilon]. 
	%	\end{align}
	
	%And we will substitute $\lambda_i$ for $\lambda_i^{(\beta)}$ without causing ambiguity.
	
	%	We start by introducing an auxiliary process $\tilde{\xi}_n$ from the process $\xi_n$, defined in \eqref{process of eigenvalues}. For any bounded $A\subset (0,+\infty)$,  set $A_n = n^{-6/5} A$.
	%	We can also define a thinned process $\Tilde{\xi}_{n}$ of $\xi_{n}$ obtained by only keeping the eigenvalues $\lambda_k$ such that the following two conditions hold
	%	\begin{itemize}
	%		\item[-] $\xi_{n}\left(\lambda_k+A_n\right)=1$.
	%		\item[-] There does not exists $\lambda_{\ell} \neq \lambda_k$ with $\xi_{n}\left(\lambda_{\ell}+A_n\right)=1$
	%		and $\abs{\lambda_k-\lambda_{\ell}}<\log^{-1} n$.
	%	\end{itemize} 
	%	Theorem \ref{small gaps of GbetaE} can be proved by establishing:
	%	\begin{enumerate}[label=Part \roman*.]
	%		\item The two point processes $\Upsilon_{n}(A\times \cdot)$ and $\tilde{\xi}_{n}(\cdot)$ are asymptotically equivalent. In other words, for any bounded interval $I\subset (-2+\ep,2-\ep)$, $\abs{\Upsilon_{n}(A\times I)}-\abs{\Tilde{\xi}_{n}(I) }$ converges to 0 in probability as $n\to\infty$.
	%		\item 
	%	\end{enumerate}
	
	%The part i (asymptotic equivalence) will be proved in the following lemma.
	In this section, we introduce two main lemmas that will imply  Lemma \ref{lemt} and Lemma \ref{lemf}, and thus Theorem \ref{small gaps of GbetaE} for GSE.  Recall in Section \ref{introd} that for any bounded measurable set $A\subset (0,+\infty)$,  we set $A_n = n^{-6/5} A$.  First, Lemma \ref{lemt} is the direct consequence of the following
	\begin{lem}
		%\textbf{(Small gaps do not cluster)}
		\label{No successive small gaps}
		Suppose $\lambda_1, \lambda_2,...,\lambda_n$ are the eigenvalues of GSE. For any interval $I\subset (-2+\epsilon,2-\epsilon)$, define $G_n$ to be the union of the following two sets: 
		\begin{align*}
			\left\{(\lambda_i,\lambda_{i+1},\lambda_{i+2})\in I^3: 1\leq i\leq n-2,
			\lambda_{i+1}-\lambda_{i}\in  A_n  \text{ and } 0\leq \lambda_{i+2}-\lambda_{i}\leq 2\sup\left(A_n\right) \right\},
		\end{align*}
		and 
		\begin{align*}
			\Big\{(\lambda_i,\lambda_{i+1},\lambda_{j},\lambda_{j+1})\in I^4: &1\leq i< j\leq n-1,
			\lambda_{i+1}-\lambda_{i} \in A_n,   \lambda_{j+1}-\lambda_{
				j}\in  A_n,\\
			&     \text{and }2\sup\left(A_n\right)<\lambda_j-\lambda_i<\log^{-1}n\Big\},
		\end{align*}
		respectively.  
		Then we have
		\begin{align*}
			\mathbb{P}\left(\left| G_n\right|>0 \right) \leq \mathbb{E}\left[\left| G_n\right|\right]\rightarrow 0\,\,\, \text{as}\,\,\, n\to\infty.
		\end{align*}
	\end{lem}
	Recall the definitions of the point processes $\xi_n$ and $\tilde\xi_n$ from Section \ref{introd}.  Let  $\rho_k$ and $\tilde\rho_k$ denote the $k$-point correlation functions of  $\xi_n$ and $\tilde\xi_n$, respectively.  To prove Lemma \ref{lemf}, we will derive some upper and lower bounds for the correlation functions $\tilde{\rho}_{k}$ and  show that these two bounds match up to the leading order. Then we will prove that the integral of the leading order term yields the limit in \eqref{factorial_conv}. 
	
	Specifically,  we define the set   
	\begin{equation}\label{omegak}
		\Omega_{k}:=\left\{(\lambda_1,\ldots, \lambda_k)\in (-2+\ep,2-\ep)^k: \min_{i\neq j}\abs{\lambda_i-\lambda_j}>\log^{-1} n \right\}.
	\end{equation}
	% $$
	% \tilde{\Omega}_k:=\left\{(\lambda_1,\ldots, \lambda_k)\in (-2+\ep,2-\ep)^k: \min_{i\neq j}\abs{\lambda_i-\lambda_j}>n^{-0.1}+2n^{-\frac{\beta+2}{\beta+1}} \sup(A)\right\}.
	% $$
	By the definition of $\tilde{\xi}_n$, we have 
	\begin{align}
		\tilde{\rho}_{k}(\lambda_1,\ldots, \lambda_k)\equiv0\quad\text{for}\quad(\lambda_1,\ldots, \lambda_k)\in\Omega_{k}^c.
		\label{cor from def of tilde rho}
	\end{align}
	On $\Omega_{k}$ we first have the upper bound
	\begin{equation}\label{upper}
		\begin{split}
			\tilde{\rho}_{k}(\lambda_1,\ldots, \lambda_k) \leq 	&\int_{\lambda_1+A_n}\cdots \int_{\lambda_k+A_n}\rho_{2k}(\lambda_1,x_1,...,\lambda_k,x_k)\mathrm{d}x_1\cdots\mathrm{d}x_k\\
			:=&L_{k,n}(\lambda_1,\ldots,\lambda_k).
		\end{split}	
	\end{equation}
	Given any $(\lambda_1,\ldots, \lambda_k)\in \Omega_{k}$, we define two sets 
	\begin{equation}
		T_{1,k,n}= 		T_{1,k,n}(\lambda_1,\ldots, \lambda_k)	:=\cup_{i=1}^k \left(\lambda_i, \lambda_i+2\sup\left(A_n\right)\right),
		\label{T1}
	\end{equation}
	and
	\begin{equation}
		\begin{split}T_{2,k,n}=& T_{2,k,n}(\lambda_1,\ldots, \lambda_k)\\:=&\cup_{i=1}^k\left\{(x,y): y\in x+A_n, 2\sup\left(A_n\right)<\abs{x-\lambda_i}\leq \log^{-1}n  \right\}. \end{split}
		\label{T2}
	\end{equation}
	Then on $\Omega_{k}$ we have  the lower bound
	\begin{equation}\label{lower}
		\begin{split}
			& \tilde{\rho}_k(\lambda_1,\ldots, \lambda_k) \\  \geq &
			\int_{\lambda_1+A_n}\cdots \int_{\lambda_k+A_n}\rho_{2k}(\lambda_1,x_1,...,\lambda_k,x_k)\mathrm{d}x_1\cdots\mathrm{d}x_k\\
			-&\int_{\lambda_1+A_n}\cdots \int_{\lambda_k+A_n}\int_{T_{1,k,n}}  \rho_{2k+1}(\lambda_1,x_1,...,\lambda_k,x_k,z)\mathrm{d}z\mathrm{d}x_1\cdots\mathrm{d}x_k\\
			-&\int_{\lambda_1+A_n}\cdots \int_{\lambda_k+A_n}\int_{T_{2,k,n}}\rho_{2k+2}(\lambda_1,x_1,...,\lambda_k,x_k,z_1,z_2)\mathrm{d}z_1\mathrm{d}z_2\mathrm{d}x_1\cdots\mathrm{d}x_k\\
			:=&L_{k,n}(\lambda_1,\ldots, \lambda_k)-E_{1,k,n}(\lambda_1,\ldots, \lambda_k)-E_{2,k,n}(\lambda_1,\ldots, \lambda_k).
		\end{split}
	\end{equation}
	In the following, we will simply write these terms as $L_{k,n},E_{1,k,n}$, and $E_{2,k,n}$.
	The main task of the article is to prove the following lemma to control the integration of the leading order term $L_{k,n}$
	and the error terms $E_{1,k,n}$ and $E_{2,k,n}$. 
	
	\begin{lem}\label{mainlemma}
		For any interval $I \subset (-2+\ep, 2-\ep)$, we have the convergence
		\begin{equation}\label{main}
			\lim_{n\to\infty} \int_{I^k \cap \Omega_{k}}L_{k,n} \mathrm{d}\lambda_1\cdots \mathrm{d}\lambda_k 
			= \left(\dfrac{1}{{540\pi^2}}\int_A u^{4}\mathrm{d}u\right)^k\int_{I^{k}} \prod_{i=1}^{k} \left( 2\pi\rho_{\semi}(\lambda_i)\right)^{6} \mathrm{d} \lambda_{1} \cdots \mathrm{d} \lambda_{k}.
		\end{equation}
		We also have the negligibility of the error terms
		\begin{equation}\label{1error}
			\lim_{n\to\infty}  \int_{I^k \cap \Omega_{k}}E_{1,k,n}\mathrm{d}\lambda_1\cdots \mathrm{d}\lambda_k = 0, 
		\end{equation}
		and 
		\begin{equation}\label{2error}
			\lim_{n\to\infty}  \int_{I^k \cap \Omega_{k}}E_{2,k,n}\mathrm{d}\lambda_1\cdots \mathrm{d}\lambda_k = 0.
		\end{equation}
	\end{lem}
	Lemma \ref{No successive small gaps} is a direct consequence of \eqref{1error}  and \eqref{2error} for the special case $k=1$, which, in turn, implies Lemma \ref{lemt} as mentioned earlier at the beginning of this section.
	
	Recall the relation between the correlations function and the factorial moment, we have 
	$$\mathbb{E}\left(\frac{|\tilde{\xi}_{n}(I)| !}{\left(|\tilde{\xi}_n(I)|-k\right) !}\right) =\int_{I^k}\tilde \rho_k d\lambda_1\cdots d\lambda_k=\int_{I^k\cap\Omega_k^c}\cdots +\int_{I^k\cap \Omega_k} \cdots. $$
	This together with \eqref{cor from def of tilde rho}, \eqref{upper}, \eqref{lower} and  Lemma \ref{mainlemma} will  imply Lemma \ref{lemf}. 
	
	Therefore, to prove Theorem \ref{small gaps of GbetaE} for GSE, it is sufficient  to prove Lemma \ref{mainlemma}. We will 
	prove \eqref{main}   in Section \ref{subsec:main}.
	The limits \eqref{1error} and  \eqref{2error} 
	will be established in Sections   \ref{subsec:error} and \ref{subsec:e2kn}, respectively.

	\section{Estimates of the leading order term}\label{subsec:main}
	In this section, we prove \eqref{main}. Recall that 
	$$L_{k,n}=\int_{\lambda_1+A_n}\cdots \int_{\lambda_k+A_n}\rho_{2k}(\lambda_1,x_1,...,\lambda_k,x_k)\mathrm{d}x_1\cdots\mathrm{d}x_k.$$
	% To prove \eqref{main},  we need the estimates of 
	%the correlation functions  on the set 
	By the Pfaffian structure of GSE in \eqref{anticor}, the correlation function is 
	$$
	\rho_{2 k}\left(\lambda_{1}, x_{1}, \ldots, \lambda_{k}, x_{k}\right) = \pfa \M_{4k}.
	\label{first term}
	$$
	Here, 
	$\M_{4k}$ is an anti-symmetric $4k \times 4k$ matrix:  	\begin{equation} \label{4k4k}
		\M_{4k} := \left( M_{i,j}\right)_{1\leq i,j\leq k}, \quad
		M_{i,j} =
		\begin{bmatrix}
			JK_n(\lambda_i,\lambda_j) & JK_n(\lambda_i, x_j) \\			JK_n(x_i,\lambda_j) & JK_n(x_i,x_j)
		\end{bmatrix}, 
	\end{equation}
	where each block $M_{i,j}$ is a $4\times4$ matrix.
	
	Recall the definition of $\mathcal{P}_{4k}$ in  \eqref{definition of P}, 
	we define its \emph{diagonal subset} 	$$\mathcal{D}_{4k} := \left\{\sigma\in \mathcal{P}_{4k}: \lceil\sigma_{2i-1}/4\rceil = \lceil\sigma_{2i}/4\rceil, i=1,...,2k\right\},$$
	where $\lceil a \rceil $ is the smallest integer greater than or equal to $a$.
	We can thus decompose  
	\begin{equation}
		\label{diagofdiag}
		\pfa \M_{4k}:=\pfaD \M_{4k}+\pfaO \M_{4k},
	\end{equation}
	where 
	$$
	\pfaD \M_{4k} 	=	\sum_{\sigma\in \mathcal{D}_{4k}}\mathrm{sgn}(\sigma)\M_{4k}(\sigma_1,\sigma_2)\M_{4k}(\sigma_3,\sigma_4)\cdots \M_{4k}(\sigma_{4k-1},\sigma_{4k})= \prod_{i=1}^{k} \pfa M_{i,i}
	\label{pfaffian of diag}
	$$
	and 
	$$
	\pfaO \M_{4k} 	=	\sum_{\sigma\in \mathcal{P}_{4k}\backslash \mathcal{D}_{4k}}\mathrm{sgn}(\sigma)\M_{4k}(\sigma_1,\sigma_2)\M_{4k}(\sigma_3,\sigma_4)\cdots \M_{4k}(\sigma_{4k-1},\sigma_{4k}).
	\label{pfaffian of off-diag}
	$$
	%We will compute the asymptotics of  the diagonal part $\pfaD \M_{4k}$ in Lemma \ref{lem:diag pfa}, which coincides with the limit on the right hand side of \eqref{main}. 
	%We shall   also prove that $\pfaO \M_{4k}$ makes a negligible contribution  to the integral of $L_{k,n}$ as $n\to\infty$.
	%And thus  \eqref{main} is a directly consequence of the following two lemmas.
	%We will compute the limit of the integral of  the diagonal part $\pfaD \M_{4k}$ in Lemma \ref{lem:diag pfa}, which coincides with the limit on the right-hand side of \eqref{main}. 
	%We also prove that $\pfaO \M_{4k}$ makes a negligible contribution  to the integral of $L_{k,n}$  in Lemma \ref{sssec:offd}.
	%\subsection{Proof of  equations \eqref{main}}

	\subsection{On-diagonal part $\pfaD \M_{4k}$}\label{might}
	For the on-diagonal part, we will prove 
	\begin{lem}\label{lem:diag pfa}
		Uniformly in $(\lambda_1,...,\lambda_k)\in I^k$, we have 
		$$
		\lim_{n\to\infty}\prod_{i=1}^{k}\int_{\lambda_i+A_n}\pfa M_{i,i}\mathrm{d}x_i
		=\left(\dfrac{1}{540\pi^2}\int_A u^{4}\mathrm{d}u\right)^k
		\prod_{i=1}^{k}  \left( 2\pi\rho_{\semi}(\lambda_i)\right)^{6}.
		\label{con of diag pfaffian}
		$$
	\end{lem}
	
	\begin{proof}
		For the diagonal block $ M_{i,i}$,  by \eqref{def of kernel 4}, its Pfaffian reads
		\begin{equation} \label{pfaffian of Mii}
			\begin{split}
				\pfa  {M}_{i,i} 			=& S_n(\lambda_i,\lambda_i)S_n(x_i,x_i) - S_n(\lambda_i,x_i)S_n(x_i,\lambda_i)+J_n(\lambda_i,x_i)V_n(\lambda_i,x_i) \\
				=& S_n(x_i,x_i)\int_{\lambda_i}^{x_i} V_n(\lambda_i,t)\mathrm{d}t
				-S_n(\lambda_i,x_i)\int_{\lambda_i}^{x_i} V_n(x_i,t)\mathrm{d}t\\
				&+J_n(\lambda_i,x_i)V_n(\lambda_i,x_i). 
			\end{split}
		\end{equation}
		By Lemma \ref{Convergence of the rescaling kernel} and the bound $\abs{x_i-\lambda_i} \lesssim n^{-6/5}$, we have the uniform estimates, 
		$$
		S_n(x_i,x_i)=n\rho_{\semi}(\lambda_i)+O(n^{1/2}),$$ and 
		\begin{align*}   \int_{\lambda_i}^{x_i} V_n(\lambda_i,t)\mathrm{d}t&=     \int_{0}
			^{x_i-\lambda_i}n\rho_{\semi}(\lambda_i)\Big(\partial_t \left[K_{\sin}(-2n\rho_{\semi}(\lambda_i)t)\right]+O\left(1\right)\Big)\mathrm{d}t\\
			&  = n\rho_{\semi}(\lambda_i) (K_{\sin}(2u_i)-1)+O\left(n^{-1/5}\right),\end{align*}
		where  $u_i:=n\rho_{\semi}(\lambda_i)(x_i-\lambda_i)\lesssim n^{-1/5}$ and thus
		$
		n\rho_{\semi}(\lambda_i)(K_{\sin}(2u_i)-1)\lesssim  n^{3/5}
		$.
		This will yield the estimate for the first term in \eqref{pfaffian of Mii} as follows, 
		\begin{align*}
			& S_n(x_i,x_i)\int_{\lambda_i}^{x_i} V_n(\lambda_i,t)\mathrm{d}t
			- n^2\rho^2_{\semi}(\lambda_i)\left(K_{\sin}(2u_i)-1\right)\\
			\lesssim & n \cdot n^{-1/5}+n^{1/2}\cdot n^{3/5}+n^{1/2}\cdot n^{-1/5}\lesssim n^{11/10}. 
		\end{align*}
		The estimates of the other two terms in \eqref{pfaffian of Mii} can be derived in the same way. In the end we  obtain the following uniform estimate
		$$\pfa M_{i,i}
		= n^2\rho_{\semi}^2\left(\lambda_i\right)\left(1-K^2_{\sin}(2u_i) + \partial_{u_i}[K_{\sin}(2u_i)]\int^{u_i}_0K_{\sin}(2t)\mathrm{d}t \right) +O\left(n^{11/10}\right).
		\label{uniform con of pfM of GSE} $$
		Note that the Lebesgue measure $\mathcal L({A_n})$ is  $O\left(n^{-6/5}\right)$.  Recall the sine kernel in \eqref{sinekerneldef},  if we further apply the Taylor expansion 
		$$
		K_{\sin}(t)=1- \frac{\pi^2t^2}{6}+ \frac{\pi^4t^4}{120}+O\left(t^6\right) \mbox{ as } t\to 0,$$ 
		we will complete the proof of the lemma as follows, 
		\begin{align*}
			&\lim_{n\to\infty}	\prod_{i=1}^{k}\int_{\lambda_i+A_n}\pfa  M_{i,i}\mathrm{d}x_i
			\\
			=& \lim_{n\to\infty}  \prod_{i=1}^{k}\left[\int_{n^{-1/5}\rho_{\semi}(\lambda_i)A} n\rho_{\semi}(\lambda_i)\left(
			\dfrac{16\pi^4u_i^4}{135}+O\left(u_i^5\right)\right)\mathrm{d}u_i+O\left(n^{-1/10}\right) \right] \\
			=&\left(\int_{A} \dfrac{u^4}{540\pi^2}\mathrm{d}u\right)^k \prod_{i=1}^{k} \left( 2\pi\rho_{\semi}(\lambda_i)\right)^6. 
		\end{align*}
	\end{proof}

	\subsection{Two rounds of transformations}\label{subsub:main}
	In order to bound the off-diagonal part $\pfaO \M_{4k}$, we need to bound all its entries. Note that Pfaffian is invariant under the congruent transformations, 
	we can perform the following \emph{Round 1 transformation}  to reduce the order of the entries of $\M_{4k}$ without changing its Pfaffian.
	\begin{itemize}
		\item[-] Calculate $(\lambda_{i+1}-x_{i+1})$ times the $(4i+2)$-th row, and add the result to the $(4i+1)$-th row for $i=0,...,k-1$. Then perform the same operations to the columns. 
		
		\item[-] Then subtract the $(4i+1)$-th row from the $(4i+3)$-th row, and subtract the $(4i+2)$-th row from the $(4i+4)$-th row for $i=0,...,k-1$. Then  perform the same operations to the columns. 
	\end{itemize}
	
	We shall use a superscript $(R)$ to denote the matrix after Round 1 transformation. 	Now we have,  
	$$\begin{aligned}
		\pfa \M_{4k}&=\pfa \M_{4k}^{(R)} \\& = \sum_{\sigma\in \mathcal{P}_{4k}}\mathrm{sgn}(\sigma)
		\M_{4k}^{(R)}(\sigma_1,\sigma_2)\M_{4k}^{(R)}(\sigma_3,\sigma_4)\cdots \M_{4k}^{(R)}(\sigma_{4k-1},\sigma_{4k}).\end{aligned}
	\label{pfaffian of M2knew}
	$$
	Since off-diagonal $4\times 4$ blocks have no influences on diagonal $4\times 4$ blocks during the transformation, we have
	$$
	\pfaD \M_{4k}=\prod_{i=1}^{k}\pfa  M_{i,i}=\prod_{i=1}^{k}\pfa  M^{(R)}_{i,i}=\pfaD \M_{4k}^{(R)}.
	$$
	Consequently,
	$$\label{pfaoid1}
	\pfaO \M_{4k}= \pfa \M_{4k}-\pfaD \M_{4k}= \pfa \M^{(R)}_{4k}-\pfaD \M^{(R)}_{4k}= \pfaO \M^{(R)}_{4k}.
	$$
	As will be clear later, it will be more convenient to bound $\pfaO \M_{4k}$ if we  make the following \emph{Round 2   transformation}  for $\M_{}^{(R)}$: by multiplying certain factors, we define 
	\begin{align}
		\M_{4k}^{(F)}(i,j)=: 
		\left\{\begin{aligned}
			n^{-6/5}\M_{4k}^{(R)}(i,j),&\quad &&\text{for $i,j$ both even,}\\
			n^{6/5}\M_{4k}^{(R)}(i,j),&\quad &&\text{for $i,j$ both odd,}\\
			\M_{4k}^{(R)}(i,j), &\quad &&\text{for all other $i,j$.}
		\end{aligned}\right.
		\label{operation of pfaffian}
	\end{align}
	For any $\sigma\in \mathcal{P}_{4k}$, 
	we have
	$$\begin{aligned}
		&|\{1\leq \ell \leq 2k : \sigma_{2\ell-1} \mbox{ and } \sigma_{2\ell} \mbox{ are both odd}\}|\\=&|\{1\leq \ell \leq 2k : \sigma_{2\ell-1} \mbox{ and } \sigma_{2\ell} \mbox{ are both even}\}|.
	\end{aligned}$$
	And thus $$\begin{aligned}
		&\M_{4k}^{(R)}(\sigma_1,\sigma_2)\M_{4k}^{(R)}(\sigma_3,\sigma_4)\cdots \M_{4k}^{(R)}(\sigma_{4k-1},\sigma_{4k})\\=&
		\M_{4k}^{(F)}(\sigma_1,\sigma_2)\M_{4k}^{(F)}(\sigma_3,\sigma_4)\cdots \M_{4k}^{(F)}(\sigma_{4k-1},\sigma_{4k}). 
	\end{aligned} $$
	This further implies that
	\begin{equation}\label{pfaoid2}
		\pfaD \M_{4k}^{(F)}=\pfaD \M_{4k}^{(R)}=\pfaD \M_{4k}^{},  \quad	 \pfaO \M_{4k}^{(F)}=\pfaO \M_{4k}^{(R)}=\pfaO \M_{4k}^{}.
	\end{equation} 
	%Combining \eqref{pfaoid1} and \eqref{pfaoid2}, we get
	%\begin{equation}
	%\pfaO \M_{}=\pfaO \M^{(F)}_{}.
	%\end{equation}
	Hence, it suffices to  bound $\M^{(F)}_{4k}$. For this purpose we need the following lemma, which will be used to control the integral of $\pfaO \M_{4k}$ in Subsection \ref{sssec:offd}. We first define the following subset of $I^k$,
	\begin{align} 
		\tilde \Omega_k=: \left\{(\lambda_1,x_1,\ldots, \lambda_k,x_k)\in I^{2k}:
		\begin{aligned}
			\abs{\lambda_i-\lambda_j}>&\log^{-1} n\text{ for } 1\leq i<j\leq k;\\
			x_i\in &\lambda_i+A_n \text{ for } i=1,...,k
		\end{aligned} \right \}.
		\label{domain of rho}
	\end{align}
	Then we have

	\begin{lem}\label{lem:round1}
		For $\lambda_i\in I$ and $x_i\in \lambda_i+A_n$ for $ i=1,...,k$, the diagonal block $ M_{i,i}^{(F)}$ has the upper bound, 
		\begin{align}
			\label{order of diagonal after trans for GSE}
			\abs{M_{i,i}^{(F)}} \lesssim  \begin{bmatrix}
				0 & n & n^{3/5} & n^{3/5}\\
				n &0 & n^{3/5}& n^{3/5}\\
				n^{3/5}& n^{3/5}& 0& n^{1/5} \\
				n^{3/5}& n^{3/5}& n^{1/5} &0
			\end{bmatrix}.
		\end{align}
		On $\tilde\Omega_k$ the off-diagonal part $ M^{(F)}_{i,j}(i<j)$  can be controlled by
		\begin{align}
			\abs{M^{(F)}_{i,j}}\lesssim 
			\begin{bmatrix}
				n^{6/5}\log n & n^{1/2} & n^{3/10}  & 
				n^{3/10} \\
				n^{1/2} & n^{-1/5}\log n & n^{-2/5}\log n & n^{-2/5}\log n \\
				n^{3/10}  & n^{-2/5}\log n &n^{-1/5}\log n&n^{-2/5}\log n \\
				n^{3/10}  & n^{-2/5}\log n & n^{-2/5}\log n & n^{-3/5}\log n 
			\end{bmatrix}.
			\label{2order of good off-diagonal after trans}
		\end{align}

	\end{lem}
	
	\begin{proof} 	Let $M_{i,j,11}^{(R)},M_{i,j,12}^{(R)},M_{i,j,21}^{(R)},M_{i,j,22}^{(R)}$ be the upper left, upper right, lower left and lower right $2\times 2$ blocks of $ M^{(R)}_{i,j}$ after the Round 1 transformation, respectively.  %The same notations are applied to  $ M^{(F)}_{i,j}$ after the Round 2 transformation,
		
		% Part (i):  \textbf{bounds for  diagonal blocks}.
		We first prove \eqref{order of diagonal after trans for GSE}. Recall \eqref{4k4k}, the original on-diagonal blocks are
		\begin{align*}
			{M}_{i,i}= \begin{bmatrix}
				0 &S_n(\lambda_i,\lambda_i) & J_n(\lambda_i,x_i)& S_n(x_i,\lambda_i)\\
				-S_n(\lambda_i,\lambda_i) &0 & -S_n(\lambda_i,x_i)& -V_n(\lambda_i,x_i)\\
				J_n(x_i,\lambda_i)& S_n(\lambda_i,x_i)& 0&S_n(x_i,x_i) \\
				-S_n(x_i,\lambda_i)& -V_n(x_i,\lambda_i)&-S_n(x_i,x_i)&0
			\end{bmatrix}
		\end{align*}
		After the Round 1 transformation, the (1,1)-entry of $M_{i,i,12}^{(R)}$ is 
		\begin{align*}
			\int_{x_i}^{\lambda_i}S_n(t,x_i)\mathrm{d}t - (\lambda_i-x_i) S_n(\lambda_i,x_i) = & (\lambda_i-x_i) S_n(\xi,x_i) - (\lambda_i-x_i) S_n(\lambda_i,x_i) \\
			= & -(\lambda_i-x_i)(\lambda_i-\xi) \partial_\gamma S_n(\gamma,x_i),
		\end{align*}
		where $\xi\in [\lambda_i,x_i]$ and $\gamma\in [\lambda_i,\xi]$. By \eqref{s4derough} and \eqref{s4derough2} we have 
		$$
		\sup_{\abs{x-y}\leq \sup(A_n)}\abs{\partial_x S_n(x,y)}= O\left(n^{9/5}\right) \quad\text{and}\quad \sup_{\abs{x-y}\leq \sup(A_n)}\abs{V_n(x,y)} = O\left(n^{9/5}\right).
		\label{order of S}
		$$ This  implies that the (1,1)-entry of $M_{i,i,12}^{(R)}$  is of order $O\left(n^{-3/5}\right) $.  The (1,2)-entry of $M_{i,i,12}^{(R)}$ is given by
		\begin{align*}
			S_n(x_i,\lambda_i) - S_n(\lambda_i,\lambda_i) - (\lambda_i-x_i)V_n(\lambda_i,x_i) = (x_i-\lambda_i) \partial_\xi S_n(\xi,\lambda_i)+O\left(n^{3/5}\right),
		\end{align*}
		which is of order $O\left(n^{3/5}\right)$. Similar analysis works for the (2,1) and (2,2) entries of $M_{i,i}^{(R)}$, and we
		get the estimate
		\begin{align*}
			\abs{M^{(R)}_{i,i,12}} \lesssim \begin{bmatrix}
				n^{-3/5}   & n^{3/5}\\
				n^{3/5} & n^{9/5} 
			\end{bmatrix}.
		\end{align*}
		Due to the anti-symmetry of $M_{i,i}^{(R)}$, $M^{(R)}_{i,i,21}$ has the same order estimate as above.  For $M^{(R)}_{i,i,22}$, its (1,2)-entry is
		\begin{align}
			&S_n(x_i,x_i)+S_n(\lambda_i,\lambda_i)-S_n(\lambda_i,x_i)-S_n(x_i,\lambda_i)+(x_i-\lambda_i)V_n(\lambda_i,x_i)\notag\\
			=&-\int_{\lambda_i}^{x_i}V_n(x_i,t)\mathrm{d}t+\int_{\lambda_i}^{x_i}V_n(\lambda_i,t)\mathrm{d}t+(x_i-\lambda_i)V_n(\lambda_i,x_i).
			\label{M22}
		\end{align}
		By the estimate (\ref{convergence of D}), we have
		\begin{align*}
			\left|V_n(\lambda_i,x_i) - n^2\rho^2_{\semi}(\lambda_i) \frac{\mathrm{d}}{\mathrm{d}u}[K_{\sin}(2u)]\right| = O\left(n\right),
		\end{align*}
		where $u:=n\rho_{\semi}(\lambda_i)(\lambda_i-x_i)\lesssim n^{-1/5}$.
		Then (\ref{M22}) can be written as 
		\begin{align*}
			&-2n\rho_{\semi}(\lambda_i)\int_{0}^{u} 
			\left(   \frac{\mathrm{d}}{\mathrm{d}t}[K_{\sin}(2t)]\right)\mathrm{d}t
			+  n\rho_{\semi}(\lambda_i)u
			\frac{\mathrm{d}}{\mathrm{d}u}[K_{\sin}(2u)]+ O\left(n^{-1/5}\right)\\
			=&2n\rho_{\semi}(\lambda_i)\left(1-K_{\sin}(2u)+\frac{\mathrm{d}}{\mathrm{d}u}[K_{\sin}(2u)]\right) + O\left(n^{-1/5}\right) \\
			=&\dfrac{4n\rho_{\semi}(\lambda_i)u^4}{15}+O\left(n^{-1/5}\right).
		\end{align*}
		Hence we get
		\begin{align*}
			\abs{M_{i,i,22}^{(R)}} \lesssim 
			\begin{bmatrix}
				0   & n^{1/5} \\
				n^{1/5} & 0
			\end{bmatrix}.
		\end{align*}
		Combining above estimates,   we have the following bound for the  diagonal block $M_{i,i}^{(R)}$ $(1\leq i\leq k)$  after the Round 1 transformation,
		\begin{align}
			\abs{M_{i,i}^{(R)}} \lesssim 
			\label{order of diagonal}
			\begin{bmatrix}
				0 &n & n^{-3/5}& n^{3/5}\\
				n &0 & n^{3/5}& n^{9/5}\\
				n^{-3/5}& n^{3/5}& 0&n^{1/5} \\
				n^{3/5}& n^{9/5}& n^{1/5}&0
			\end{bmatrix}.
		\end{align}
		Therefore, \eqref{order of diagonal after trans for GSE} follows from \eqref{order of diagonal}
		and the definition of Round 2 transformation \eqref{operation of pfaffian}
		
		Now we prove \eqref{2order of good off-diagonal after trans}. 
		Recall \eqref{4k4k} again, the original off-diagonal blocks are
		\begin{align*}
			M_{i,j} = \begin{bmatrix}
				J^{}_n(\lambda_i,\lambda_j)&S^{}_n(\lambda_j,\lambda_i) & J^{}_n(\lambda_i,x_j)& S^{}_n(x_j,\lambda_i)\\
				-S^{}_n(\lambda_i,\lambda_j) &-V^{}_n(\lambda_i,\lambda_j) & -S^{}_n(\lambda_i,x_j)& -V^{}_n(\lambda_i,x_j)\\
				J^{}_n(x_i, \lambda_j)& S^{}_n(\lambda_j,x_i)&J^{}_n(x_i, x_j) &S^{}_n(x_j,x_i) \\
				-S^{}_n(x_i,\lambda_j)&-V^{}_n(x_i,\lambda_j)&-S^{}_n(x_i,x_j)& -V^{}_n(x_i,x_j)
			\end{bmatrix}.
		\end{align*}% But there are two cases we need to consider separately.
		% We say $\bar M_{i,j}^{(\beta)},i< j$ is a "good" block, if 
		% $$\lambda_i<x_i<\lambda_j<x_j, \quad\text{or}\quad\lambda_j<x_j<\lambda_i<x_i.$$
		% Otherwise it is called a "bad" block. The reason for this separation is the last term $\varepsilon(x-y)$ of $J^{(1)}(x,y)$. This means that $\bar M_{i,j}^{(4)},i< j$ has the same order in both cases, but the "bad" $\bar M_{i,j}^{(1)},i< j$ has the different result.
		% \begin{itemize}
		%     \item[-] Order of the "good" off-diagonal blocks
		% \end{itemize}
		We consider the upper left $2\times 2$ block $M_{i,j,11}^{(R)}$ of $M^{(R)}_{i,j}$ after the Round 1 transformation.
		Its entries are  
		$$
		\begin{aligned}
			M_{i,j,11}^{(R)}(1,1)=&
			J_n^{}(\lambda_i,\lambda_j)-(\lambda_i-x_i)S_n^{}(\lambda_i,\lambda_j)+(\lambda_j-x_j) S_n^{}(\lambda_j,\lambda_i)\\
			&-(\lambda_j-x_j)(\lambda_i-x_i)V_n^{}(\lambda_i,\lambda_j),\\
			M_{i,j,11}^{(R)}(1,2)=&S_n^{}(\lambda_j,\lambda_i)-(\lambda_i-x_i)V_n^{}(\lambda_i,\lambda_j), \\
			M_{i,j,11}^{(R)}(2,1)=&-(\lambda_j-x_j)V_n^{}(\lambda_i,\lambda_j)-S_n^{}(\lambda_i,\lambda_j), \\
			M_{i,j,11}^{(R)}(2,2)=&-V_n^{}(\lambda_i,\lambda_j).
		\end{aligned}
		$$
		%	Let $d_{i,j} := \abs{\lambda_i-\lambda_j}$, and
		%	$$\alpha_{i,j}:=\min\left\{\abs{\lambda_j-\lambda_i}, \abs{\lambda_j-x_i},\abs{x_j-\lambda_i}, \abs{x_i-x_j}\right\}.$$
		The requirement that $x_i\in \lambda_i+A_n$ implies $$\abs{\lambda_i-x_i}\leq \sup(A_n)(=\sup\{x:x\in A_n\})\lesssim n^{-6/5}.$$ In addition,
		using \eqref{domain of rho} where  $(\lambda_1,\ldots, \lambda_k) \in \tilde\Omega_{k}$, for all $i\neq j$, we have
		$$
		\min\left\{\abs{\lambda_j-\lambda_i}, \abs{\lambda_j-x_i},\abs{x_j-\lambda_i}, \abs{x_i-x_j}\right\} \geq \log^{-1} n-2\sup(A_n) \geq \frac{1}{2}\log^{-1} n . 
		$$ %	$\alpha_{i,j}\leq d_{i,j}\leq 2\alpha_{i,j}$. %This shows that $\alpha_{i,j}$ and $d_{i,j}$ are equivalent as $n\rightarrow\infty$. Therefore, we will use $d_{i,j}$ as a uniform representation of the four terms of $\alpha_{i,j}$.
		Then by the estimates (\ref{s4rough}), (\ref{s4derough2}) and \eqref{intbd},
		we get
		\begin{align*}
			\abs{M_{i,j,11}^{(R)}} \lesssim  \begin{bmatrix}
				\log n  & n^{1/2}  \\
				n^{1/2} & n\log n 
			\end{bmatrix}.
		\end{align*}
		%	Here we have used the fact that so that $\min\{n, 1/\log^{-1}n\}=\log n$ 
		Other parts of $M^{(R)}_{i,j}$ can be bounded in  the same manner; therefore, we omit the details. This leads us to the final estimate,  \begin{align}
			\abs{M^{(R)}_{i,j}} \lesssim  
			\begin{bmatrix}
				\log n & n^{1/2} & n^{-9/10}  & 
				n^{3/10} \\
				n^{1/2} & n\log n & n^{-2/5 }\log n & n^{4/5}\log n \\
				n^{-9/10}  & n^{-2/5}\log n &n^{-7/5}\log n & n^{-2/5}\log n \\
				n^{3/10}  & n^{4/5}\log n & n^{-2/5}\log n& n^{3/5}\log n 
			\end{bmatrix}.
			\label{2order of good off-diagonal}
		\end{align}
		Now  \eqref{2order of good off-diagonal after trans} follows from \eqref{2order of good off-diagonal} and the definition of the Round 2 transformation \eqref{operation of pfaffian}.

	\end{proof} 
	
	% \begin{itemize}
	%     \item[-] Order of the "bad" off-diagonal blocks 
	% \end{itemize}
	
	% In fact, the "bad" off-diagonal block $\bar M^{(4)}_{i,j}, i<j$ has the same result as \eqref{2order of good off-diagonal}. For the "bad" $\bar M^{(1)}_{i,j}, i<j$, by the definition, 
	% \begin{align}
	%     \lambda_i\leq \lambda_j\leq x_i \quad \text{or} \quad \lambda_j\leq \lambda_i\leq x_j.
	%     \label{bad order}
	% \end{align}
	% Then $\abs{\lambda_i-\lambda_j} \in A^{(\beta)}_n$, and $\d_{i,j} = n$. So we can simplify the term $O\left(q_{i,j}\right)$ in \eqref{2order of good off-diagonal} to $O(n)$.
	% In addition, because of the function $\epsilon(x-y)$ in $J^{(1)}_n(x,y)$ and the "bad" order \eqref{bad order},
	% we can only bound
	% $M_{i,j,12}^{(1)}(1,1)$, $M_{i,j,21}^{(1)}(1,1)$ and $M_{i,j,22}^{(1)}(1,1)$ by $O(1)$. 
	
	% Hence for the "bad" block $\bar M^{(1)}_{i,j}, i<j$, we have the estimation
	% \begin{align}
	%     \left(\begin{array}{cccc}
	%       O\left(\log n\right) & O(n) & O(1) & 
	%       O\left(n^{1/2}\right)\\
	%     O(n) & O(n^2) & O(1) & O\left(n^{3/2}\right) \\
	%     O(1) & O(1)&O(1)&O(1) \\
	%     O\left(n^{1/2}\right) & O\left(n^{3/2}\right)& O(1)& O(n)
	%     \end{array}\right).
	%     \label{2order of bed off-diagonal}
	% \end{align}
	
	% 	Then the uniform convergence \eqref{con of diag pfaffian} holds for $\beta=1,4$.
	
	\subsection{Off-diagonal part $\pfaO \M_{4k}$}
	\label{sssec:offd}
	%Now we consider the sum of the terms containing off-diagonal elements of $\M_{4k}$ and prove 
	Next, we prove the following lemma, which indicates that the contribution from the off-diagonal blocks is negligible.		\begin{lem}\label{lem:bdoffdiag}    
		Uniformly in $(\lambda_1,...,\lambda_k)\in \Omega^k$, we have 
		\begin{equation*} 
			\lim_{n\to\infty}\int_{\lambda_1+A_n}\cdots \int_{\lambda_k+A_n} \abs{\pfaO \M_{4k}}\mathrm{d}x_1\cdots \mathrm{d}x_k =0.
		\end{equation*}
	\end{lem}  %	Given any $\sigma\in \mathcal{P}_{4k}$, define 
	%	\begin{align}
	%		\Lambda_{\sigma}^{(\beta)}:=\left\{a^{(\beta)}_{\sigma_{2i-1}\sigma_{2i}}:i=1,2,...,2k\right\}.
	%		\label{definition of sigma}
	%	\end{align}
	%	There are two subsets of $\Lambda^{(\beta)}_{\sigma}$ we are more concerned,
	%	\begin{align*}
	%		\Lambda^{(\beta)}_{\sigma,\text{odd}}:=\left\{a^{(\beta)}_{i,j}\in\Lambda^{(\beta)}_{\sigma}:i \text{ and } j \text{ both odd}\right\},\quad
	%		\Lambda^{(\beta)}_{\sigma,\text{even}}:=\left\{a^{(\beta)}_{i,j}\in\Lambda^{(\beta)}_{\sigma}:i \text{ and } j \text{ both even}\right\}.
	%	\end{align*}
	%Observe that this will not change the value of any term of (\ref{pfaffian of off-diag}), since $\abs{\Lambda^{(\beta)}_{\sigma,\text{odd}}} = \abs{\Lambda^{(\beta)}_{\sigma,\text{even}}}$ for any $\sigma\in \mathcal{P}_{4k}$.  After this transform  matrix $\M_{4k}^{(R)}$ is changed into a new matrix which shall be denoted by $\M_{4k}^{(F)}$.
	%	For any $(\lambda_i,...,\lambda_k)\in \Omega_{k,n}$, $\abs{\lambda_i-\lambda_j}>n^{-0.1}>2\sup\left(A^{(\beta)}_n\right)$, the estimations \eqref{order of diagonal}, \eqref{order of diagonal of GOE} and \eqref{2order of good off-diagonal} still hold true.
	%	With the estimations above, we will prove a stronger estimation than (\ref{cons of g}).
	We first construct a $4k\times 4k$ matrix $\U_{4k}$  in the block form 
	$$
	\U_{4k}=\left(U_{i,j}\right)_{1\leq i, j\leq k},
	$$
	where  $U_{i,i}$ and $U_{i,j}$  are   $4\times 4$ upper bound matrices   given by  the right-hand side of \eqref{order of diagonal after trans for GSE} and \eqref{2order of good off-diagonal after trans}, respectively; that is, $$|\M_{4k}^{(F)}(i,j)|\lesssim \U_{4k}(i,j).$$  For example, 
	$\U_{4k}(i,j)=n^{6/5}\log n$ if  $i\equiv j\equiv 1$ (mod 4) and  $i\neq j.$ 
	We set
	\begin{equation}\label{def of U}
		U_{4k}\left(\sigma\right):=\prod_{i=1}^{2k}  \U_{4k}(\sigma_{2i-1},\sigma_{2i}). 
	\end{equation} Let $\mathbf{\Lambda}_{4k}$ be a 0-1  matrix of size $4k\times 4k$ corresponding to $\sigma$:  
	\begin{equation}\label{defSigma}
		\mathbf{\Lambda}_{4k}(i,j)=1   \Leftrightarrow  \exists \, \ell  \mbox{ s.t. } (\sigma_{2\ell-1}, \sigma_{2\ell})=(i,j).
	\end{equation}
	We rewrite $\mathbf{\Lambda}_{4k}=(\Lambda_{i,j})_{1\leq i, j\leq k}$ where each $\Lambda_{i,j}$ is a $4\times 4$ block.  We now  partition the non-zero entries of $\mathbf{\Lambda}_{4k}$ into diagonal regimes (denoted by a letter `$D$') and off-diagonal regimes (denoted by a letter `$O$'). The partition is performed according to the growth order of $n$ as described in \eqref{order of diagonal after trans for GSE} and \eqref{2order of good off-diagonal after trans}. The diagonal part contains 4 regimes:
	\begin{align*}
		D_1:=&\left\{(i,j)\in [4k]\times [4k]:
		\lceil i/4\rceil = 
		\lceil j/4\rceil, \mathbf{\Lambda}_{4k}(i,j)=1, i\equiv 1(\mod 4),\,j\equiv 2(\mod 4)\right\},\\
		D_2:=&\left\{(i,j)\in [4k]\times [4k]:
		\lceil i/4\rceil = 
		\lceil j/4\rceil, \mathbf{\Lambda}_{4k}(i,j)=1, i\equiv 1(\mod 4),\,j\equiv 3,4(\mod 4)\right\},\\
		D_3:=&\left\{(i,j)\in [4k]\times [4k]:
		\lceil i/4\rceil = 
		\lceil j/4\rceil, \mathbf{\Lambda}_{4k}(i,j)=1, i\equiv 2(\mod 4),\,j\equiv 3,4(\mod 4) \right\},\\
		D_4:=&\left\{(i,j)\in [4k]\times [4k]:
		\lceil i/4\rceil = 
		\lceil j/4\rceil, \mathbf{\Lambda}_{4k}(i,j)=1, i\equiv 3(\mod 4)\right\}, 
	\end{align*}
	where $[4k]:=\{1, 2,..., 4k\}$.

	Figure \ref{fig1} illustrates the partition for each $4\times 4$ diagonal  blocks $\Lambda_{i,i}, i=1,.., k$, providing the information on  which nonzero entries belong to each $D_i$. For example, for the block $\Lambda_{1,1}$,  $(1,2)\in D_1$ if $\mathbf{\Lambda}_{4k}(1,2)=1$;  $(1,3)\in D_2$ if $\mathbf{\Lambda}_{4k}(1,3)=1$, etc. Note that we only consider and partition entries in the upper triangle of $\mathbf{\Lambda}_{4k}$, since the entries in the lower triangle  of $\mathbf{\Lambda}_{4k}$  are all 0. Figure \ref{fig2} shows the partition information for $4\times 4$ off-diagonal   blocks in the upper triangle, defined similarly to the $D_i$'s.  To avoid redundancy, their definitions are omitted. We also let
	$$O_4=O_{4,1}\cup O_{4,2}, \,\, O_5=O_{5,1}\cup O_{5,2}.$$
	\begin{center}

		\begin{figure}[!htbp]
			\centering
			\begin{minipage}{.45\textwidth}
				\centering
				\begin{tikzpicture}
					% 绘制矩阵
					\matrix (m) [matrix of math nodes, nodes={minimum size=5mm, anchor=center}, column sep=3mm, row sep=3mm, right delimiter={]}, left delimiter={[}]{
						0 & \phantom{a_{12}} & \phantom{a_{13}} & \phantom{a_{14}} \\
						& 0 & \phantom{a_{23}} & \phantom{a_{24}} \\
						&  & 0 & \phantom{a_{34}} \\
						&  &  & 0 \\
					};
					
					% 绘制方框和标注
					\draw[dashed, line width=0.8pt, rounded corners] ([xshift=-0.9mm,yshift=0.9mm]m-1-2.north west) rectangle ([xshift=0.9mm,yshift=-0.9mm]m-1-2.south east) node[midway] {$D_1$};
					\draw[dashed, line width=0.8pt, rounded corners] ([xshift=-0.9mm,yshift=0.9mm]m-1-3.north west) rectangle ([xshift=0.9mm,yshift=-0.9mm]m-1-4.south east) node[midway] {$D_2$};
					\draw[dashed, line width=0.8pt, rounded corners] ([xshift=-0.9mm,yshift=0.9mm]m-2-3.north west) rectangle ([xshift=0.9mm,yshift=-0.9mm]m-2-4.south east) node[midway] {$D_3$};
					\draw[dashed, line width=0.8pt, rounded corners] ([xshift=-0.9mm,yshift=0.9mm]m-3-4.north west) rectangle ([xshift=0.9mm,yshift=-0.9mm]m-3-4.south east) node[midway] {$D_4$};
				\end{tikzpicture}
				\caption{Partition  of $4\times 4$ diagonal  blocks.}
				\label{fig1}
			\end{minipage}
			
			\begin{minipage}{.45\textwidth}
				\centering
				\begin{tikzpicture}
					% 绘制矩阵
					\matrix (m) [matrix of math nodes, nodes={minimum size=5mm, anchor=center}, column sep=3mm, row sep=3mm, right delimiter={]}, left delimiter={[}]{
						\phantom{a_{11}} & \phantom{a_{12}} & \phantom{a_{13}} & \phantom{a_{14}} \\
						\phantom{a_{21}} & \phantom{a_{22}} & \phantom{a_{23}} & \phantom{a_{24}} \\
						\phantom{a_{31}} & \phantom{a_{32}} & \phantom{a_{33}} & \phantom{a_{34}} \\
						\phantom{a_{41}} & \phantom{a_{42}} & \phantom{a_{43}} & \phantom{a_{44}} \\
					};
					
					% 绘制方框和标注
					\draw[dashed, line width=0.8pt, rounded corners] ([xshift=-0.9mm,yshift=0.9mm]m-1-1.north west) rectangle ([xshift=0.9mm,yshift=-0.9mm]m-1-1.south east) node[midway] {$O_1$};
					\draw[dashed, line width=0.8pt, rounded corners] ([xshift=-0.9mm,yshift=0.9mm]m-1-2.north west) rectangle ([xshift=0.9mm,yshift=-0.9mm]m-1-2.south east) node[midway] {$O_2$};
					\draw[dashed, line width=0.8pt, rounded corners] ([xshift=-0.9mm,yshift=0.9mm]m-2-1.north west) rectangle ([xshift=0.9mm,yshift=-0.9mm]m-2-1.south east) node[midway] {$O_2$};
					\draw[dashed, line width=0.8pt, rounded corners] ([xshift=-0.9mm,yshift=0.9mm]m-1-3.north west) rectangle ([xshift=0.9mm,yshift=-0.9mm]m-1-4.south east) node[midway] {$O_3$};
					\draw[dashed, line width=0.8pt, rounded corners] ([xshift=-0.9mm,yshift=0.9mm]m-3-1.north west) rectangle ([xshift=0.9mm,yshift=-0.9mm]m-4-1.south east) node[midway] {$O_3$};
					% O4,1 和 O4,2
					\draw[dashed, line width=0.8pt, rounded corners] ([xshift=-0.9mm,yshift=0.9mm]m-2-2.north west) rectangle ([xshift=0.9mm,yshift=-0.9mm]m-2-2.south east) node[midway] {$O_{4,1}$};
					\draw[dashed, line width=0.8pt, rounded corners] ([xshift=-0.9mm,yshift=0.9mm]m-3-3.north west) rectangle ([xshift=0.9mm,yshift=-0.9mm]m-3-3.south east) node[midway] {$O_{4,2}$};
					% O5,1
					\draw[dashed, line width=0.8pt, rounded corners] ([xshift=-0.9mm,yshift=0.9mm]m-2-3.north west) rectangle ([xshift=0.9mm,yshift=-0.9mm]m-2-4.south east) node[midway] {$O_{5,1}$};
					\draw[dashed, line width=0.8pt, rounded corners] ([xshift=-0.9mm,yshift=0.9mm]m-3-2.north west) rectangle ([xshift=0.9mm,yshift=-0.9mm]m-4-2.south east) node[midway] {$O_{5,1}$};
					% O5,2
					\draw[dashed, line width=0.8pt, rounded corners] ([xshift=-0.9mm,yshift=0.9mm]m-3-4.north west) rectangle ([xshift=0.9mm,yshift=-0.9mm]m-3-4.south east) node[midway] {$O_{5,2}$};
					\draw[dashed, line width=0.8pt, rounded corners] ([xshift=-0.9mm,yshift=0.9mm]m-4-3.north west) rectangle ([xshift=0.9mm,yshift=-0.9mm]m-4-3.south east) node[midway] {$O_{5,2}$};
					% O6
					\draw[dashed, line width=0.8pt, rounded corners] ([xshift=-0.9mm,yshift=0.9mm]m-4-4.north west) rectangle ([xshift=0.9mm,yshift=-0.9mm]m-4-4.south east) node[midway] {$O_6$};
				\end{tikzpicture}
				\caption{Partition of $4\times 4$ off-diagonal  blocks.}
				\label{fig2}
			\end{minipage}
		\end{figure}
	\end{center}
	Given $\sigma\in \mathcal{P}_{4k}\backslash \mathcal{D}_{4k}$,
	by the definitions of $\U_{4k}$ and $U_{4k}(\sigma)$ and Lemma \ref{lem:round1}, we get
	\begin{equation}\label{usigma1}
		U_{4k}(\sigma)\lesssim  \left(\log^Cn\right) n^{\ord_1}, 
	\end{equation}
	where  
	\begin{equation}
		\label{ord1}
		\begin{split}
			\ord_1:= & |D_1|+\dfrac{3}{5}|D_2|+\dfrac{3}{5}|D_3|+\dfrac{1}{5}|D_4|\\
			& 			+\frac{6}{5}|O_1|+\frac{1}{2}\abs{O_2} 
			+\frac{3}{10}\abs{O_{3} }-\frac{1}{5}\abs{O_{4,1} }-\frac{1}{5}\abs{O_{4,2} }-\frac{2}{5}\abs{O_5}-\frac{3}{5}|O_6|.
			%,\\	\sum_{i=1}^{k-1}\left(\mu_i+\nu_i\right) =|O_2|+|O_4|+|O_5|+|O_6|.
		\end{split}
	\end{equation}
	Interestingly, the following lemma shows that $\text{Ord}_1$ for any off-diagonal permutation is always less than $6k/5$, which is exactly the factor needed to derive the negligibility of off-diagonal blocks  as in Lemma \ref{lem:bdoffdiag}. 
	
	\begin{lem}\label{lem:ord1bd}
		For any $\sigma\in \mathcal{P}_{4k}\backslash \mathcal{D}_{4k}$, it holds that
		\begin{equation}\label{ord1bd}
			\ord_1<6k/5.
		\end{equation}
	\end{lem}
	Assuming Lemma \ref{lem:ord1bd}, we can complete the proof of Lemma \ref{lem:bdoffdiag}. 
	\begin{proof}[Proof of Lemma \ref{lem:bdoffdiag}]
		%For the ease of notations we introduce a set
		%$$
		%\bar \Omega_{k,n}=\left\{(\lambda_1,x_1,\ldots,\lambda_k,x_k ):(\lambda_1,\ldots, \lambda_k) \in \Omega_{k,n}, x_i \in \lambda_i+A_n, \forall \, 1\leq i\leq k\right\},
		%$$
		Recall the set $\tilde\Omega_k$ in \eqref{domain of rho},  its Lebesgue measure satisfies 
		$$
		\mathcal L({\tilde \Omega_{k}})\lesssim \mathcal L({A_n})^{k}\lesssim n^{-6k/5}. 
		$$
		This, together with \eqref{ord1bd}, implies that  for any $\sigma\in \mathcal{P}_{4k}\backslash \mathcal{D}_{4k}$,
		$$	\label{cons of upper bound}
		\begin{aligned}
			&	\int_{\tilde\Omega_{k}} 	\abs{\M_{4k}^{(F)}(\sigma_1,\sigma_2)\M_{4k}^{(F)}(\sigma_3,\sigma_4)\cdots \M_{4k}^{(F)}(\sigma_{4k-1},\sigma_{4k})}\mathrm{d}\lambda_1\mathrm{d}x_1 \cdots \mathrm{d}\lambda_k\mathrm{d}x_k\\
			\lesssim&	\int_{\tilde \Omega_{k}}U_{4k}(\sigma) \mathrm{d}\lambda_1\mathrm{d}x_1 \cdots \mathrm{d}\lambda_k\mathrm{d}x_k\\
			\lesssim  & \left(\log^C n\right)  n^{-6k/5+\ord_1},
		\end{aligned} 
		$$
		which converges to 0 as $n\to\infty$. 
		Here, we have  used the fact that 
		$\abs{\M_{4k}^{(F)}(i,j)} \lesssim \U_{4k}(i,j) $ for all $1\leq i,j\leq 4k$. This implies $$ \lim_{n\to\infty}\int_{\lambda_1+A_n}\cdots \int_{\lambda_k+A_n} \abs{\pfaO \M_{4k}^{(F)}}\mathrm{d}x_1\cdots \mathrm{d}x_k =0, $$which completes the proof by \eqref{pfaoid2}. 
	\end{proof}
	Now we prove Lemma \ref{lem:ord1bd}.
	\begin{proof}[Proof of Lemma \ref{lem:ord1bd}]
		The starting point of the proof is to derive some identities which hold for all $\sigma\in \mathcal{P}_{4k}$. Since the two sets $\{\sigma_i,1\leq i\leq 4k\}$ and $[4k]$ are the same, if we group $\sigma_i's$ according to their congruence classes modulo 4, then for each $\ell=0,1,2,3$, 
		\begin{equation}\label{sigma_mod}
			\abs{\{ i\in [4k]:\sigma_i\equiv \ell (\mod 4)\}}= \abs{\{i\in [4k]:i\equiv \ell(\mod 4)\}} = k.
		\end{equation}
		We can also rewrite the left hand hand of \eqref{sigma_mod} via
		\begin{equation}\label{decomp1}
			\begin{split}
				\abs{\{ i\in [4k]:\sigma_i\equiv \ell (\mod 4)\}}
				=&\abs{  \{ (i,j): \mathbf{\Lambda}_{4k}(i,j)=1, i\equiv \ell, j\not\equiv \ell (\mathrm{mode}\, 4)  \} }\\
				&+\abs{  \{ (i,j): \mathbf{\Lambda}_{4k}(i,j)=1, i\not\equiv \ell, j\equiv \ell (\mathrm{mode}\, 4)  \} }\\
				&+2\abs{  \{ (i,j): \mathbf{\Lambda}_{4k}(i,j)=1, i\equiv j\equiv \ell  (\mathrm{mode}\, 4)  \} }.
			\end{split}
		\end{equation}
		Consider first the  case $\ell=1$. By the definition of the sets $D_i's$ and $O_i's$, we see that sum of the first and second lines on  the right hand side  of 
		\eqref{decomp1} is equal to $\abs{D_1}+\abs{O_2}+\abs{O_3}$, while the third lines are $2\abs{O_1}$. It follows from this observation and \eqref{sigma_mod} (with $\ell=1$) that
		\begin{equation}\label{cons1}
			2|O_1|+|O_2|+|O_{3} |+|D_1|+|D_2| =k.
		\end{equation}
		Similarly, for $\ell=2$, we get
		\begin{equation}\label{cons2}
			|O_2|+2|O_{4,1} |+|O_{5,1}|+|D_1|+|D_3|=k. 
		\end{equation}
		Now, combining the cases of $\ell=3$ and $\ell=4$, we get
		\begin{equation}\label{cons3}
			|O_{3} |+2|O_{4,2} |+|O_{5,1}|+2|O_{5,2}|+2|O_6|+|D_2|+|D_3|+2|D_4|=2k.
		\end{equation}
		We will need one more identity. Since the number of nonzero  $\mathbf{\Lambda}_{4k}(i,j)'s$ is $2k$, we have
		\begin{equation}\label{cons4}
			\abs{D_1}+\abs{D_2}+\abs{D_3}+\abs{D_4}+\abs{O_1}+\abs{O_2}+|O_3|+\abs{O_4}+\abs{O_5}+\abs{O_6}=2k.
		\end{equation}
		By computing the linear combination
		$$
		\frac{3}{5} \times \eqref{cons1}
		+\frac{3}{5} \times \eqref{cons2}+\frac{1}{5}\times \eqref{cons3}
		-\frac{1}{5}\times  \eqref{cons4},
		$$
		we find that 	\begin{equation}\label{addid}
			\begin{split}
				\frac{6k}{5}=	&|D_1|+\frac{3}{5}|D_2|+\frac{3}{5}|D_3|+\frac{1}{5}|D_4|+
				|O_1|+|O_2| \\
				&+\frac{3}{5}|O_{3} |+\abs{O_{4,1} }+\frac{1}{5}|O_{4,2} |
				+\frac{3}{5}|O_{5,1}|+\frac{1}{5}|O_{5,2}|+\frac{1}{5}|O_6|.
			\end{split}
		\end{equation}
		Comparing \eqref{ord1} and \eqref{addid}, we see that the assertion  $\ord_1<6k/5$  for $\sigma\in \mathcal{P}_{4k}/\mathcal{D}_{4k}$ is equivalent to
		\begin{equation}\label{equicond}
			\abs{O_1}<\frac{3}{2}\abs{O_{3} }
			+2\abs{O_{4,2}} 
			+ \frac 52 \abs{O_2}
			+3\abs{O_{5,2}}+4\abs{O_6}   +5\abs{O_{5,1}}
			+			6\abs{O_{4,1} }.
		\end{equation}
		Actually, we shall prove a stronger conclusion  for $\sigma\in \mathcal{P}_{4k}$:
		\begin{equation}\label{o1eq1}
			\abs{O_1}\leq \abs{O_2}+|O_3|+\abs{O_4}+\abs{O_5}+\abs{O_6}.
		\end{equation}
		Indeed, assuming that \eqref{o1eq1} holds, then \eqref{equicond} follows by combining \eqref{o1eq1} and the trivial fact $
		\sum_{i=1}^6\abs{O_i}>0$ for $\sigma\in \mathcal{P}_{4k}/\mathcal{D}_{4k}$. 
		
		We now provide an intuitive explanation of \eqref{o1eq1}. Suppose that $\mathbf{\Lambda}_{4k}(1,5)=1$, i.e., $(1,5)\in O_1$, then there must exist at least one element $b\in \{2,3,4\}$ and  $b'>5$ such that $\mathbf{\Lambda}_{4k}(b,b')=1$, leading to $(b,b') \in \cup_{i=2}^6 O_i$. This suggests that the right hand side of \eqref{o1eq1} should dominate  $\abs{O_1}$.  
		
		To prove \eqref{o1eq1} formally, we observe that,  by the construction of $\mathbf{\Lambda}_{4k}$, for each $1\leq i\leq 4k$,
		\begin{equation}
			\sum_{j:j\neq i} \left(\mathbf{\Lambda}_{4k}(i,j)+\mathbf{\Lambda}_{4k}(j,i)\right)=1.
		\end{equation}
		Summing this equality over all $4\ell -3\leq i\leq 
		4\ell $ for some $1\leq \ell \leq k$, and using the fact that $\mathbf{\Lambda}_{4k}(i,j)=0$ for $i\geq j$, we get
		\begin{equation}\label{sigsum}
			2     \sum_{\lceil  i/4 \rceil = \lceil j/4\rceil=\ell}\mathbf{\Lambda}_{4k}(i,j)+ \sum_{\lceil  i/4 \rceil =\ell < \lceil j/4\rceil}\mathbf{\Lambda}_{4k}(i,j)+
			\sum_{   \lceil j/4\rceil< \lceil  i/4 \rceil =\ell}\mathbf{\Lambda}_{4k}(j,i)=4. 
		\end{equation}%where $\lceil a \rceil$ is the smallest integer that is $\geq a$. 
		We now define a function $ h_{\sigma}$ as follows:
		\begin{equation}\label{hsigmadef}
			h_{\sigma}(\ell):=\sum_{\lceil  i/4 \rceil =\ell < \lceil j/4\rceil}\mathbf{\Lambda}_{4k}(i,j)+
			\sum_{ \lceil j/4\rceil< \lceil  i/4 \rceil =\ell }\mathbf{\Lambda}_{4k}(j,i), \, 1\leq \ell \leq k.
		\end{equation}
		By \eqref{sigsum}, $h_{\sigma}(\ell)$ is an even integer for all $\ell$.  In particular,  we have the implications
		\begin{equation}\label{hi}\begin{split}
				&\forall \, 1\leq \ell_1<\ell_2\leq k,
				(4\ell_1-3,4\ell_2-3)\in O_1  \Rightarrow  \mathbf{\Lambda}_{4k}(4\ell_1-3,4\ell_2-3)=1
				\\   & \Rightarrow   \min\{h_{\sigma}(\ell_1),h_{\sigma}(\ell_2)\} \geq 1
				\Rightarrow   \min\{h_{\sigma}(\ell_1),h_{\sigma}(\ell_2)\} \geq 2,
			\end{split}
		\end{equation}
		where we used the evenness of $h_{\sigma}(\ell_1)$ and $h_{\sigma}(\ell_2)$ in the last step. 
		Consider the set 
		\begin{equation}\label{hato1def}
			\widehat{O}_1:=\{1 \leq \ell\leq k: 
			\exists \ell'\in [1,k], \mbox{ s.t. }  (4\min\{\ell,\ell'\}-3, 4\max\{\ell,\ell'\}-3) \in O_1\}.
		\end{equation}
		It follows from  \eqref{hsigmadef} and \eqref{hi} that
		\begin{equation}\label{hnsum}
			\sum_{1\leq \ell\leq k} h_{\sigma}(\ell)\geq    \sum_{\ell\in \widehat{O}_1} h_{\sigma}(\ell)\geq 2 \abs{\widehat{O}_1}=4|O_1|. 
		\end{equation}
		On the other hand, we also have
		\begin{equation}\label{hnsum2}
			\begin{split}
				\sum_{1\leq \ell\leq k} h_{\sigma}(\ell)&=
				\sum_{1\leq \ell \leq k}\sum_{\lceil  i/4 \rceil =\ell < \lceil j/4\rceil}\mathbf{\Lambda}_{4k}(i,j)+
				\sum_{1\leq \ell \leq k}\sum_{ \lceil j/4\rceil< \lceil  i/4 \rceil =\ell }\mathbf{\Lambda}_{4k}(j,i)\\
				&=\sum_{\lceil  i/4 \rceil < \lceil j/4\rceil} \mathbf{\Lambda}_{4k}(i,j)+\sum_{\lceil  j/4 \rceil < \lceil i/4\rceil} \mathbf{\Lambda}_{4k}(j,i)=2\sum_{j=1}^6 \abs{O_j}. 
			\end{split}
		\end{equation}
		Now \eqref{o1eq1}  follows by combining \eqref{hnsum} and \eqref{hnsum2}.  And thus we completes the proof of Lemma \ref{lem:ord1bd}. 
	\end{proof}
	
	\subsection{Proof of  \eqref{main}: the limit of the main term $L_{k,n}$}\label{may}
	Recall the definition of $L_{k,n}$ in \eqref{upper}, and the decomposition of the correlation function $\rho_{2k}=\pfaD \M_{4k}+\pfaO \M_{4k}$ in \eqref{diagofdiag}, we can
	%We decompose $\rho_{2k,n}$ into diagonal part $\rho_{2k,n,D}$ and $\rho_{2k,n,O}$, i.e., 
	%$$
	%\rho_{2k,n,D}=\prod_{i=1}^k \pfa M_{i,i}, \quad \rho_{2k,n,O}=\rho_{2k,n}-\rho_{2k,n,D}.
	%$$
	%Similarly, 
	decompose $L_{k,n}$ into  $L_{k,n,D}$ and $L_{k,n,O}$ by integrating $\pfaD \M_{4k}$ and $\pfaO \M_{4k}$, respectively. 
	
	By Lemma \ref{lem:diag pfa} and the fact that the measure of the set $I^k \cap \Omega_{k}^c$ converges to 0,  we have 	\begin{align*}&	\lim_{n\to\infty} \int_{I^k \cap \Omega_{k}}L_{k,n,D} \mathrm{d}\lambda_1\cdots \mathrm{d}\lambda_k \\ =& 	\int_{I^k}\Big(\lim_{n\to\infty} \prod_{i=1}^{k}\int_{\lambda_i+A_n}\pfa M_{i,i}\mathrm{d}x_i\Big)\mathrm{d}\lambda_1\cdots \mathrm{d}\lambda_k\\
		= &\left(\dfrac{1}{540\pi^2}\int_A u^{4}\mathrm{d}u\right)^k \int_{I^k}\left( 2\pi\rho_{\semi}(\lambda_i)\right)^{6} \mathrm{d} \lambda_{1} \cdots \mathrm{d} \lambda_{k}.\end{align*}
	On the other side,
	by Lemma \ref{lem:bdoffdiag}, we have
	\begin{align*}	\label{main_2}
		&	\lim_{n\to\infty}\Big| \int_{I^k \cap \Omega_{k}}L_{k,n,O} \mathrm{d}\lambda_1\cdots \mathrm{d}\lambda_k\Big|\\
		\leq &\lim_{n\to\infty}  \int_{I^k \cap \Omega_{k}}\int_{\lambda_1+A_n}\cdots \int_{\lambda_k+A_n} \abs{\pfaO \M_{4k}}\mathrm{d}x_1\cdots \mathrm{d}x_k \mathrm{d}\lambda_1\cdots \mathrm{d}\lambda_k=0, 
	\end{align*}
	which  completes the proof of \eqref{main}. %now follows from \eqref{main_1}
	%and \eqref{main_2}. 
	%	\leq & \sum_{\sigma\in \mathcal{P}_{4k}\backslash \mathcal{D}_{4k} }\lim_{n\to\infty} \int_{\tilde \Omega_{k}} 	\abs{\M^{(R)}(\sigma_1,\sigma_2)\M^{(R)}(\sigma_3,\sigma_4)\cdots \M^{(R)}(\sigma_{4k-1},\sigma_{4k})}\mathrm{d}\lambda_1\mathrm{d}x_1 \cdots \mathrm{d}\lambda_k\mathrm{d}x_k\\ =&0, 
	
	\section{Estimates of the error terms}
	\label{sec:err} Recall the definitions of $E_{1,k,n}$ and $E_{2,k,n}$ in \eqref{lower}, we now prove  \eqref{1error} and \eqref{2error}. This will complete the proof of Lemma \ref{mainlemma}, and subsequently, Theorem \ref{small gaps of GbetaE} for GSE.

	\subsection{Estimates of $E_{1,k,n}$} 
	\label{subsec:error}
	%\subsubsection{Entry-wise estimates on the transformed correlation matrix}
	To bound $E_{1,k,n}$, we need
	to control the $(2k+1)$-point correlation function $\rho_{2k+1,n}(\lambda_1,x_1,\ldots, \lambda_k,x_k,z)$ where $z\in T_{1,k,n}(\lambda_1,\ldots, \lambda_k)$ (recall \eqref{T1}). 
	We denote $$
	T_{1,k,n}^{[i]}:= \left(\lambda_i,\lambda_i+2\sup(A_n)\right).
	$$  Without loss of generality, we may further assume  $z\in T_{1,k,n}^{[k]}$. Then we have 
	$$\rho_{2k+1,n}(\lambda_1,x_1,\ldots, \lambda_k,x_k,z)=\pfa  \M_{4k+2},$$
	which is  a $(4k+2)\times (4k+2)$ matrix that can be written in a block form
	$$
	\M_{4k+2}=\left(M_{i,j} \right)_{1\leq i,j\leq k}.
	$$
	Here, for $i,j<k$,  $M_{i,j}$ has size $4\times 4$, while the sizes for $M_{i,k}$ ($i<k$) and $M_{k,k}$ are  $4\times 6$ and $6\times 6$, respectively. 
	
	As before we make two rounds of transformations to $\M_{4k+2}$. 
	We first make the same transformations  as in Section \ref{subsec:main} to the first $4k-4$ rows and columns of $\M_{4k+2}$.
	Then we perform the following additional operations to the last 6 rows/columns of $\M_{4k+2}$:
	\begin{itemize}
		\item[-] Add $(\lambda_k-x_k)$ times of the $(4k-2)$-th row to the $(4k-3)$-th row.  Also subtract $(x_k-z)$ times of the $(4k+2)$-th row from $(4k+1)$-th row. Then perform the same operations to the columns. 
		\item[-] Subtract the $(4k-3)$-th row from the $(4k-1)$-th  and $(4k+1)$-th rows, and also subtract the $(4k-2)$-th row from the $4k$-th and $(4k+2)$-th rows. Then also perform the same operations to the columns. 
	\end{itemize}
	
	After these row/column transformations are done, we  make the same operation as in \eqref{operation of pfaffian}.  And we denote the new matrix by $\M_{4k+2}^{(F)}$ as before.

	Then the upper-left  minor of size $4k-4$ in $\M_{4k+2}^{(F)}$   become the same as in Section \ref{subsec:main}, since they are not affected by the last 6 rows/columns in the transformations.
	Hence  the diagonal  blocks $ M^{(F)}_{i,i}$ ($1\leq i\leq k-1$) and off-diagonal blocks $M_{i,j}$ ($i,j<k$) have the same bound given by  \eqref{order of diagonal after trans for GSE} and \eqref{2order of good off-diagonal after trans}, respectively.
	Moreover,  we can show that the last diagonal  block $M^{(F)}_{k,k}$ have the bound
	\begin{equation}\label{m3es GSE}
		\abs{M^{(F)}_{k,k}} \lesssim 
		\begin{bmatrix}
			0 & n &  n^{3/5} & n^{3/5} & n^{3/5} &
			n^{3/5} \\
			n & 0 & n^{3/5} & n^{3/5} & n^{3/5} & n^{3/5}\\
			n^{3/5}	& n^{3/5} & 0 & n^{1/5} & n^{3/5} &n^{3/5} \\
			n^{3/5}	& n^{3/5} & n^{1/5} & 0 & n^{3/5} & n^{3/5} \\
			n^{3/5}	& n^{3/5} & n^{3/5} & n^{3/5} & 0 & n^{1/5} \\
			n^{3/5}	& n^{3/5} & n^{3/5} & n^{3/5} & n^{1/5} & 0
		\end{bmatrix}.
	\end{equation}
	For the off-diagonal part, it remains to control $M^{(F)}_{i,k}$ for $i<k$.  We decompose  
	$$
	M^{(F)}_{i,k}=\left( M^{(F)}_{i,k,1},M^{(F)}_{i,k,2}\right),
	$$ 
	where $M^{(F)}_{i,k,1}$ is a square matrix consisting of the first 4 rows/columns of $M^{(F)}_{i,k}$, and $M^{(F)}_{i,k,2}$ is the remaining part with size $4\times 2$. Here $M^{(F)}_{i,k,1}$ has the same bound as \eqref{2order of good off-diagonal after trans}. For $M^{(F)}_{i,k,2}$, we can show that
	\begin{align}
		\abs{M^{(F)}_{i,k,2}} \lesssim 
		\begin{bmatrix}
			n^{3/10} & 	n^{3/10} \\
			n^{-2/5}\log n &   	n^{-2/5}\log n\\
			n^{-1/5}\log n &  	n^{-2/5}\log n\\
			n^{-2/5}\log n &  	n^{-3/5}\log n
		\end{bmatrix}.
		\label{extra}
	\end{align}
	Using \eqref{m3es GSE} and \eqref{extra}, we can 
	construct a $(4k+2)\times (4k+2)$ matrix $\U_{4k+2}$ that serves as an entrywise upper bound for $\M^{(F)}_{4k+2}$.
	
	Now we can bound the diagonal and off-diagonal parts of $E_{1,k,n}$ in a unified way.  Recall that for any $\sigma\in \mathcal{P}_{4k+2}$, we  associate it with a 0-1 matrix $\mathbf{\Lambda}_{4k+2}$ via \eqref{defSigma}. We now introduce  additional diagonal sets $D_{5}$ and $D_6$ (which lie in the on diagonal block where the indices of entries are $[4k-3,4k+2]\times [4k+1,4k+2]$, which is the last $6\times6$ diagonal block), and off-diagonal sets $O_7,O_8,O_9,O_{10}$ (which lie in $4\times2$ blocks where the indices of entries are $[4i-3,4i]\times [4k+1,4k+2]$, $i=1,..,k-1$).  Together they constitute the last two columns of $\mathbf{\Lambda}_{4k+2}$. Their positions are shown in the  Figure \ref{fig:diagonal_block} and Figure \ref{fig:off_diagonal_blocks} below, respectively.

	\begin{figure}[!htbp]
		\centering
		
		\begin{minipage}[b]{0.47\textwidth}
			\centering
			\begin{tikzpicture}[baseline=(m.south)]
				% 绘制矩阵
				\matrix (m) [matrix of math nodes, nodes={minimum size=5mm, anchor=center}, column sep=3mm, row sep=3mm, right delimiter={]}, left delimiter={[}]{
					0 & \phantom{a_{12}} & \phantom{a_{13}} & \phantom{a_{14}} & \phantom{a_{15}} & \phantom{a_{16}} \\
					\phantom{a_{21}} & 0 & \phantom{a_{23}} & \phantom{a_{24}} & \phantom{a_{25}} & \phantom{a_{26}} \\
					\phantom{a_{31}} & \phantom{a_{32}} & 0 & \phantom{a_{34}} & \phantom{a_{35}} & \phantom{a_{36}} \\
					\phantom{a_{41}} & \phantom{a_{42}} & \phantom{a_{43}} & 0 & \phantom{a_{45}} & \phantom{a_{46}} \\
					\phantom{a_{51}} & \phantom{a_{52}} & \phantom{a_{53}} & \phantom{a_{54}} & 0 & \phantom{a_{56}} \\
					\phantom{a_{61}} & \phantom{a_{62}} & \phantom{a_{63}} & \phantom{a_{64}} & \phantom{a_{65}} & 0 \\
				};
				
				% 绘制方框和标注
				\draw[dashed, line width=0.8pt, rounded corners] ([xshift=-0.9mm,yshift=0.9mm]m-1-2.north west) rectangle ([xshift=0.9mm,yshift=-0.9mm]m-1-2.south east) node[midway] {$D_1$};
				\draw[dashed, line width=0.8pt, rounded corners] ([xshift=-0.9mm,yshift=0.9mm]m-1-3.north west) rectangle ([xshift=0.9mm,yshift=-0.9mm]m-1-4.south east) node[midway] {$D_2$};
				\draw[dashed, line width=0.8pt, rounded corners] ([xshift=-0.9mm,yshift=0.9mm]m-2-3.north west) rectangle ([xshift=0.9mm,yshift=-0.9mm]m-2-4.south east) node[midway] {$D_3$};
				\draw[dashed, line width=0.8pt, rounded corners] ([xshift=-0.9mm,yshift=0.9mm]m-3-4.north west) rectangle ([xshift=0.9mm,yshift=-0.9mm]m-3-4.south east) node[midway] {$D_4$};
				\draw[dashed, line width=0.8pt, rounded corners] ([xshift=-0.9mm,yshift=0.9mm]m-1-5.north west) rectangle ([xshift=0.9mm,yshift=-0.9mm]m-1-5.south east) node[midway] {$D_{5,1}$};
				\draw[dashed, line width=0.8pt, rounded corners] ([xshift=-0.9mm,yshift=0.9mm]m-1-6.north west) rectangle ([xshift=0.9mm,yshift=-0.9mm]m-1-6.south east) node[midway] {$D_{5,2}$};
				\draw[dashed, line width=0.8pt, rounded corners] ([xshift=-0.9mm,yshift=0.9mm]m-2-5.north west) rectangle ([xshift=0.9mm,yshift=-0.9mm]m-2-5.south east) node[midway] {$D_{5,2}$};
				\draw[dashed, line width=0.8pt, rounded corners] ([xshift=-0.9mm,yshift=0.9mm]m-2-6.north west) rectangle ([xshift=0.9mm,yshift=-0.9mm]m-2-6.south east) node[midway] {$D_{5,3}$};
				\draw[dashed, line width=0.8pt, rounded corners] ([xshift=-0.9mm,yshift=0.9mm]m-3-5.north west) rectangle ([xshift=0.9mm,yshift=-0.9mm]m-4-5.south east) node[midway] {$D_{5,4}$};
				\draw[dashed, line width=0.8pt, rounded corners] ([xshift=-0.9mm,yshift=0.9mm]m-3-6.north west) rectangle ([xshift=0.9mm,yshift=-0.9mm]m-4-6.south east) node[midway] {$D_{5,5}$};
				\draw[dashed, line width=0.8pt, rounded corners] ([xshift=-0.9mm,yshift=0.9mm]m-5-6.north west) rectangle ([xshift=0.9mm,yshift=-0.9mm]m-5-6.south east) node[midway] {$D_{6}$};
			\end{tikzpicture}
			\caption{Partition of the last $6\times 6$ diagonal block .}
			\label{fig:diagonal_block}
		\end{minipage}

		\begin{minipage}[b]{0.35\textwidth}
			\centering
			\begin{tikzpicture}[baseline=(m.south)]
				% 绘制矩阵
				\matrix (m) [matrix of math nodes, nodes={minimum size=5mm, anchor=center}, column sep=3mm, row sep=3mm, right delimiter={]}, left delimiter={[}]{
					\phantom{a_{11}} & \phantom{a_{12}} \\
					\phantom{a_{21}} & \phantom{a_{22}} \\
					\phantom{a_{31}} & \phantom{a_{32}} \\
					\phantom{a_{41}} & \phantom{a_{42}} \\
				};
				
				% 绘制方框和标注
				\draw[dashed, line width=0.8pt, rounded corners=0.9mm] ([xshift=-0.9mm,yshift=0.9mm]m-1-1.north west) rectangle ([xshift=0.9mm,yshift=-0.9mm]m-1-1.south east) node[midway] {$O_{7,1}$};
				\draw[dashed, line width=0.8pt, rounded corners=0.9mm] ([xshift=-0.9mm,yshift=0.9mm]m-1-2.north west) rectangle ([xshift=0.9mm,yshift=-0.9mm]m-1-2.south east) node[midway] {$O_{7,2}$};
				\draw[dashed, line width=0.8pt, rounded corners=0.9mm] ([xshift=-0.9mm,yshift=0.9mm]m-2-1.north west) rectangle ([xshift=0.9mm,yshift=-0.9mm]m-2-1.south east) node[midway] {$O_{9,1}$};
				\draw[dashed, line width=0.8pt, rounded corners=0.9mm] ([xshift=-0.9mm,yshift=0.9mm]m-2-2.north west) rectangle ([xshift=0.9mm,yshift=-0.9mm]m-2-2.south east) node[midway] {$O_{9,2}$};
				\draw[dashed, line width=0.8pt, rounded corners=0.9mm] ([xshift=-0.9mm,yshift=0.9mm]m-3-1.north west) rectangle ([xshift=0.9mm,yshift=-0.9mm]m-3-1.south east) node[midway] {$O_{8}$};
				\draw[dashed, line width=0.8pt, rounded corners=0.9mm] ([xshift=-0.9mm,yshift=0.9mm]m-3-2.north west) rectangle ([xshift=0.9mm,yshift=-0.9mm]m-3-2.south east) node[midway] {$O_{9,3}$};
				\draw[dashed, line width=0.8pt, rounded corners=0.9mm] ([xshift=-0.9mm,yshift=0.9mm]m-4-1.north west) rectangle ([xshift=0.9mm,yshift=-0.9mm]m-4-1.south east) node[midway] {$O_{9,4}$};
				\draw[dashed, line width=0.8pt, rounded corners=0.9mm] ([xshift=-0.9mm,yshift=0.9mm]m-4-2.north west) rectangle ([xshift=0.9mm,yshift=-0.9mm]m-4-2.south east) node[midway] {$O_{10}$};
			\end{tikzpicture}
			\caption{Partition of $4\times 2$ off-diagonal block.}
			\label{fig:off_diagonal_blocks}
		\end{minipage}
		
	\end{figure}
	% \end{center}
	
	We define $
	U_{4k+2}(\sigma) 
	$ as in \eqref{def of U}
	(we need to replace $2k$ by $2k+1$ in the upper limit of the product).
	Similarly to \eqref{usigma1}for all $(\lambda_1,\ldots, \lambda_k)\in \Omega$ and $\abs{z-\lambda_k}<2\sup(A_n)$, we have
	\begin{align}
		\label{def of ord3}
		U_{4k+2}(\sigma)\lesssim  \left(\log^C n\right)n^{\ord_2},
	\end{align}
	where
	\begin{equation}\label{ord32}
		\begin{split}
			\ord_2:=\ord_1+\frac{3}{5}|D_5|+\frac{1}{5}|D_6|
			+\frac{3}{10}|O_7|
			-\frac{1}{5}|O_8|-\frac{2}{5}|O_9|-\frac{3}{5}|O_{10}|,
			%   &	|D_1|+\dfrac{3}{5}|D_2|+\dfrac{3}{5}|D_3|+\dfrac{1}{5}|D_4|
			%		+\frac{6}{5}|O_1|+\frac{1}{2}\abs{O_2} 
			%		+\frac{3}{10}\abs{O_{3}}-\frac{1}{5}|O_4| 
			%		-\frac{2}{5}\abs{O_{5}}\\
			%  &-\frac{3}{5}|O_6|+\frac{3}{10}|O_7|
			% -\frac{1}{5}|O_8|-\frac{2}{5}|O_9|-\frac{3}{5}|O_{10}|+\frac{3}{5}|D_5|+\frac{1}{5}|D_6|.
		\end{split}
	\end{equation}
	and $\ord_1$ is given in \eqref{ord1}. 
	As before, we will show
	\begin{equation}\label{ord3<}
		\ord_2<{6(k+1)}/{5}.
	\end{equation}
	Indeed, assuming \eqref{ord3<} and using the definition of $\tilde{\Omega}_k$ in \eqref{domain of rho}, we have 
	\begin{equation}
		\begin{split}
			& \int_{I^k \cap \Omega_k}E_{1,k,n}(\lambda_1,\ldots, \lambda_k).	 \mathrm{d}\lambda_1\cdots \mathrm{d}\lambda_k\\
			\lesssim &
			\sum_{\sigma\in \mathcal{P}_{4k+2}}\sum_{i=1}^k \int_{\tilde{ \Omega}_k}\int_{T_{1,k,n}^{[i]}}
			U_{4k+2}(\sigma)\mathrm{d}z \mathrm{d}\lambda_1 \mathrm{d}x_1\cdots   \mathrm{d}\lambda_k \mathrm{d}x_k\\
			\lesssim &  \left(\log^C n\right) n^{\ord_2}
			\mathcal{L}(A_n)^{k}\cdot \sup(A_n)\\
			\lesssim &\left(\log^C n\right)n^{\ord_2-6(k+1)/5},
		\end{split}
	\end{equation}
	which vanishes as $n\to\infty$, as desired. This will complete  \eqref{1error}.
	
	It remains to  establish \eqref{ord3<}. 
	Analogously  to \eqref{cons1} - \eqref{cons4}, we have 
	\begin{equation}\label{2cons1}
		2|O_1|+|O_2|+|O_{3}|+|D_1|+|D_2|+W_1 =k+1,
	\end{equation}
	\begin{equation}\label{2cons2}	|O_2|+2|O_{4,1}|+|O_{5,1}|+|D_1|+|D_3|+W_2=k+1,
	\end{equation}
	\begin{equation}\label{2cons3}
		|O_{3} |+2|O_{4,2} |+|O_{5,1}|+2|O_{5,2}|+2|O_6|+|D_2|+|D_3|+2|D_4|+W_3 =2k,
	\end{equation}
	\begin{equation}\label{2cons4}
		\sum_{i=1}^6 |D_i|+\sum_{i=1}^{10}\abs{O_i}=2k+2, 
	\end{equation}
	where 
	\begin{equation*}
		\begin{split}
			W_1:&=2|O_{7,1}|+|O_8|+|O_{9,1}|+|O_{9,4}|+
			2|D_{5,1}|+|D_{5,2}|+|D_{5,4}|+|D_6|,\\
			W_2:&=|O_{7,2}|+2|O_{9,2}|+|O_{9,3}|+|D_{5,2}|+2|D_{5,3}|+|D_{5,5}|+|D_6|,\\
			W_3:&=|O_8|+|O_{9,3}|+|O_{9,4}|+|O_{10}|+|D_{5,4}|+|D_{5,5}|.
		\end{split} 
	\end{equation*}
	Computing the linear combination$$
	\frac{3}{5} \times \eqref{2cons1}+\frac{3}{5}\times \eqref{2cons2}+\frac{1}{5}\times \eqref{2cons3}-\frac{1}{5}\times \eqref{2cons4},
	$$ we get 
	\begin{equation*}%\label{addid2}
		\begin{split}
			\frac{6k+5}{5}=	&|D_1|+\frac{3}{5}|D_2|+\frac{3}{5}|D_3|+\frac{1}{5}|D_4|+
			|O_1|+|O_2| +\frac{3}{5}|O_{3}|+\abs{O_{4,1}}+\frac{1}{5}|O_{4,2}| \\
			&
			+\frac{3}{5}|O_{5,1}|+\frac{1}{5}|O_{5,2}|+\frac{1}{5}|O_6|+W_4,
		\end{split}
	\end{equation*}
	where 
	\begin{equation}
		\begin{split}
			W_4:= &|O_{7,1}|+\frac{2}{5}|O_{7,2}|+\frac{3}{5}|O_8|+\frac{2}{5}|O_{9,1}|+|O_{9,2}|+ \frac{3}{5}|O_{9,3}|+\frac{3}{5}|O_{9,4}|\\
			&+ |D_{5,1}|+|D_{5,2}|+|D_{5,3}|+\frac{3}{5}|D_{5,4}|+\frac{3}{5}|D_{5,5}|+|D_6|. 
		\end{split}
	\end{equation}
	The condition \eqref{ord3<} is thus equivalent to
	\begin{equation}\label{comp1}
		\begin{split}
			|O_1|<&1+\frac{5}{2}\abs{O_2}+\frac{3}{2}\abs{O_{3}}+
			6\abs{O_{4,1}}+2\abs{O_{4,2}}+5\abs{O_{5,1}}+3\abs{O_{5,2}}+4\abs{O_{6}}\\
			&+7|O_{7,1}|+\frac{1}{2}|O_{7,2}|+
			4|O_8|+4|O_{9,1}|+7|O_{9,2}|+5|O_{9,3}|+
			5|O_{9,4}|+3|O_{10}|\\
			&+2|D_{5,1}|+2|D_{5,2}|+2|D_{5,3}|.
		\end{split}
	\end{equation}
	Clearly \eqref{comp1} holds if $\abs{O_1}=0$. Thus we only consider the case $\abs{O_1}>0$. 
	Repeating the argument around \eqref{o1eq1}, we get
	\begin{equation}\label{o110}
		\abs{O_1}\leq \sum_{i=2}^{10} |O_i|. 
	\end{equation}
	Then \eqref{comp1}  follows from \eqref{o110} and the fact $\abs{O_{7,2}}\leq  1$ which implies $|O_{7,2}|< 1+\abs{O_{7,2}}/2.$
	Therefore, we establish \eqref{ord3<}.

	\subsection{Estimates of $E_{2,k,n}$}\label{subsec:e2kn} 
	% 	For \eqref{2error}, we can bound $E^{(\beta)}_{2k+2,n}$ by $L^{(\beta)}_{2k+2,n}$. 
	
	Given $(\lambda_1,...,\lambda_k)\in I^{k}\cap \Omega_{k}$, let $T_{3,k,n}$ be the first margin of $T_{2,k,n}$,
	\begin{align*}
		T_{3,k,n}=T_{3,k,n}(\lambda_1,\ldots, \lambda_k):=\cup_{i=1}^k\left\{z_1:2\sup\left(A_n\right)<\abs{z_1-\lambda_i}\leq \log^{-1} n\right\}.
	\end{align*}
	Then we have the relation
	\begin{align}
		\int_{I^k \cap \Omega_{k}}E_{2,k,n}\mathrm{d}\lambda_1\cdots \mathrm{d}\lambda_k=\int_{I^k \cap \Omega_{k}}\mathrm{d}\lambda_1\cdots \mathrm{d}\lambda_k\int_{T_{3,k,n}} L_{k+1,n}(\lambda_1,...,\lambda_k,z_1)  \mathrm{d}z_1,
	\end{align}
	where $L_{k+1,n}$ is defined by \eqref{upper}. Similarly to the decomposition of $L_{k,n}$ in Section \ref{may}, we decompose $L_{k+1,n}(\lambda_1,...,\lambda_k,z_1)$ into $L_{k+1,n,D}$ and $L_{k+1,n,O}$. %Note that  $L_{k+1,n,O}$  corresponds to $\pfa \M_{4k+4}$, where $\M_{4k+4}$ is a $(4k+4)\times (4k+4)$ matrix, which can also be viewed as a block matrix containing $(k+1)^2$ blocks, each having size $4\times 4$. %as we did for $L_{k,n}$.
	By Lemma \ref{lem:diag pfa}, $L_{k+1,n,D}$ is uniformly bounded on $I^{k+1}$. Since the Lebesgue measure  $\mathcal L(T_{3,k,n})$ converges to 0 as $n\rightarrow \infty$, we get
	\begin{align}
		\int_{I^k \cap \Omega_{k}}\mathrm{d}\lambda_1\cdots \mathrm{d}\lambda_k\int_{T_{3,k,n}} L_{k+1,n,D}(\lambda_1,...,\lambda_k,z_1)  \mathrm{d}z_1 \rightarrow 0.
	\end{align}
	% 	For $L^{(\beta)}_{k+1,n,O}$, it is easy to see
	%	\begin{align}
	%		\begin{aligned}
	%			\int_{I^k \cap \Omega_{k,n}}\mathrm{d}\lambda_1\cdots \mathrm{d}\lambda_k\int_{T^{(\beta)}_{3}} L^{(\beta)}_{k+1,n,O}  \mathrm{d}z_1
	%			\leq \int_{I^{k+1} \cap \Omega_{k+1,n}} L^{(\beta)}_{k+1,n,O}\mathrm{d}\lambda_1\cdots \mathrm{d}\lambda_{k+1}.
	%		\end{aligned}
	%	\end{align}
	% 	By \eqref{main_2}, we complete the proof of \eqref{2error}.
	We now turn to estimate the off-diagonal part $L_{k+1,n,O}$. 
	Since $\abs{\lambda_i-\lambda_j}\geq \log^{-1} n$ on $\Omega_{k}$, we can rewrite $T_{3,k,n}=\cup_{i=1}^k T_{3,k,n}^{[i]}$ where 
	\begin{equation}\label{t3i}
		\begin{split}
			T_{3,k,n}^{[i]}=\left\{z_1:2\sup\left(A_n\right)<\abs{z_1-\lambda_i}\leq \log^{-1} n;\, \abs{z_1-\lambda_j}\geq \frac{1}{2}\log^{-1} n, \forall \, j\neq i\right\}.
		\end{split}
	\end{equation}
	Without loss of generality, we may only consider $T_{3,k,n}^{[1]}$ here. 
	Similarly to $\M_{4k}$ defined in \eqref{4k4k}, we let $\M_{4k+4}$  be the $(4k+4)\times (4k+4)$ kernel matrix of the $2k+2$ points  given in the order of $\lambda_1,x_1, z_1, z_2, \lambda_2,x_2,\ldots, \lambda_k,x_k$.

	We perform the two rounds of transformations on $\M_{4k+4}$ as in Section \ref{subsub:main}, and write the final transformed matrix in a block form $$\M_{4k+4}^{(F)}=\left(M_{i,j}^{(F)}\right)_{1\leq i,j\leq k+1}.$$   %(This time we need to repeat the column/row operations for $k+1$ times instead of $k$, since we have $k+1$ diagonal blocks.) We also modify the entries according to  \eqref{operation of pfaffian}. Again we denote the new matrix using a superscript $(F)$ and write it in a block form $$\M_{4k+4}^{(F)}=\left(M_{i,j}^{(F)}\right)_{1\leq i,j\leq k+1}.$$  
	By the definition of $T_{3, k, n}^{[1]}$ in \eqref{t3i}, for $i\geq 2$, we have $\abs{z_1-\lambda_i}>(\log^{-1}n)/2$. Thus, for $(i,j)\neq (1,2)$, $M^{(F)}_{i,j}$ has the same bounds \eqref{order of diagonal after trans for GSE} for  $i=j$ and \eqref{2order of good off-diagonal after trans} for $i\neq j$.
	To express the estimate for $M^{(F)}_{1,2}$, 
	we define two function $f_n$ and $\bar f_n$   by 
	$$
	f_n(\lambda_1,z_1):=  \min \left\{n,\frac{1}{\abs{\lambda_1-z_1}}
	\right\}, \bar f_n:= f_n+n^{1/2}.
	$$
	Inspecting the proof of \eqref{2order of good off-diagonal after trans} in Lemma \ref{lem:round1}, we see that the off-diagonal block $ M^{(F)}_{1,2}$ can be bounded by
	\begin{align}\label{badbd}	M^{(F)}_{1,2}\lesssim \begin{bmatrix}
			n^{6/5}\log n & \bar f_n(\lambda_1,z_1) & n^{-1/5}\bar f_n(\lambda_1,z_1)  & 
			n^{-1/5}\bar f_n(\lambda_1,z_1) \\
			\bar f_n(\lambda_1,z_1) & n^{-1/5}f_n(\lambda_1,z_1) & n^{-2/5}  f_n(\lambda_1,z_1) & n^{-2/5} f_n(\lambda_1,z_1) \\
			n^{-1/5}\bar f_n(\lambda_1,z_1)  & n^{-2/5}f_n(\lambda_1,z_1) &n^{-1/5}f_n(\lambda_1,z_1)&n^{-2/5}f_n(\lambda_1,z_1) \\
			n^{-1/5}\bar f_n(\lambda_1,z_1)  & n^{-2/5}f_n(\lambda_1,z_1)&n^{-2/5}f_n(\lambda_1,z_1)&n^{-3/5}f_n(\lambda_1,z_1) 
		\end{bmatrix}.
	\end{align} We will not give a detailed proof of \eqref{badbd}, but rather point out that \eqref{badbd}  is related to \eqref{2order of good off-diagonal after trans}  by replacing  $\log n$ and   $n^{1/2}$  with $f_n$ and $\bar f_n$, respectively. In deriving \eqref{badbd} 
	we have also used the condition $2\sup(A_n)<\abs{z_1-\lambda_1}$ so that 
	$$
	\min\left\{\abs{\lambda_1-z_1}, \abs{\lambda_1-z_2},\abs{x_1-z_1}, \abs{x_1-z_2}\right\} \geq \frac{   2\abs{\lambda_1-z_1}  }{2}.
	$$
	As in Section \ref{sssec:offd}, we construct a $(4k+4) \times (4k+4)$ matrix $\U_{4k+4}$ which dominates $\M_{4k+4}$ in the entrywise sense. 
	We define $U_{4k+4}(\sigma)$ and $\mathbf{\Lambda}_{4k+4}$   similarly to \eqref{def of U} and \eqref{defSigma}, respectively. 
	
	Note  that \eqref{badbd} indicates that terms of the form 
	\begin{equation}\label{prodab}
		\int_{\lambda_1-\log^{-1}n}^{\lambda_1+\log^{-1} n} \left(f_n(\lambda_1,z_1)^a \times \bar f_n(\lambda_1, z_1)^b\right)\mathrm{d}z_1
	\end{equation}
	will appear in  the integral of $U_{4k+4}(\sigma)$,  for certain  integers $a,b\geq 0$.
	(The relevant combinations of $(a,b)$ will be indicated in Table \ref{table:kab} below.) 
	
	To get the upper bound for \eqref{prodab}, we let $\kappa(a,b)$ be the unique non-negative real number which satisifies the following inequality for some $C>0$,
	\begin{equation}\label{kab}
		\log^{-C} n
		\lesssim\left(\int_{-\log^{-1}n}^{\log^{-1}n} 
		(n^{1/2}+ \min\{n, \frac{1}{x}\} )^{a} \min\{n, \frac{1}{x}\}^{b}\mathrm{d}x \right) n^{-(\kappa(a,b)	+a/2)} 
		\lesssim \log^C n.
	\end{equation}
	%In other words, $\kappa(a,b)$ gives the polynomial exponent for the ratio of \eqref{prodab} to $n^{a/2}$. %(We define $\kappa$ in this way to make comparisons easier with the bounds in the last section.) 	
	A direct computation gives the  Table \ref{table:kab} below.  The combinations (1,3) and (2,3) are not needed so their values $\kappa(a,b)$ are omitted.
	\begin{table}[htbp]
		% \label{table:kab}
		\begin{center}
			\begin{tabular}{ |c|c|c|c|c| }
				\hline
				& $b=0$ & $b=1$ & $b=2$ & $b=3$ \\ 
				\hline
				$a=0$ & 0 & 0 & 1& 2 \\ 
				\hline
				$a=1$ & 0 & ${1}/{2}$ & ${3}/{2}$ & $\backslash$ \\ 
				\hline
				$a=2$ & 0 & 1 & 2& $\backslash$ \\
				\hline
				%	\caption{The value of $\kappa(a,b)$}
			\end{tabular}
			\vspace{1pt}
			\caption{The value of $\kappa(a,b)$ for relevant pairs $(a,b)$. }
			\label{table:kab}
		\end{center}
	\end{table} 
	
	%We now introduce a few more notations to express the upper bound$$
	%U_{4k+4}(\sigma)=\prod_{i=1}^{2k+2}  \U_{4k+4}(\sigma_{2i-1},\sigma_{2i})
	%.$$ 
	For sets like $O_{1}$ and $O_{2} $ defined in the last section,
	we use a bar to denote their restriction to the  off-diagonal $4\times 4$ block $(\mathbf{\Lambda}_{4k+4}(i,j) )_{1\leq i\leq 4, 5\leq j\leq 8}$, and use a prime to denote its complement. For instance, 
	\begin{equation}\label{exo1}
		\bar O_{1}=O_{1}\cap \left\{(i,j): 
		\lceil i/4\rceil=1, \, \lceil j/4 \rceil =2 
		\right\}, \quad  O_1'=O_1\backslash \bar O_1. 
	\end{equation}
	We define two quantities $Y_1(\sigma)$ and $Y_2(\sigma)$, corresponding to the exponents $a$ and $b$ in \eqref{prodab},
	\begin{equation}\label{defy1y2}
		\begin{split}
			Y_1(\sigma)&=\abs{\bar O_2}+\abs{\bar O_{3} },\\
			Y_2(\sigma)&=
			\abs{\bar O_{4} }+\abs{\bar O_5}+\abs{\bar O_6}.
		\end{split}
	\end{equation}
	The comparison between \eqref{2order of good off-diagonal after trans} and \eqref{badbd}  as well as the definition of the function $\kappa(\cdot, \cdot)$ makes it clear that, for some constant  $C>0$,
	\begin{equation}
		\int_{z_1 \in T^{[1]}_{3,k,n}}U_{4k+4}(\sigma)\mathrm{d}z_1 \lesssim
		\left(\log^C n\right)   n^{\ord_3}, 
	\end{equation}
	where
	\begin{equation}
		\ord_3=\ord_1+\kappa(Y_1(\sigma),
		Y_2(\sigma)),
	\end{equation}
	and $\ord_1$ takes the same expression as \eqref{ord1}.  
	
	Recall from Section \ref{sssec:offd} that in our analysis of off-diagonal part of $L_{k,n}$, we established $\ord_1<6k/5$ to  ensure the associated error converges to zero. Here, since we have $(k+1)$ pairs of small gaps, in order to prove \eqref{2error}, it suffices to show
	\begin{equation}\label{ord2bdd}
		\ord_3<6(k+1)/5.
	\end{equation}
	As in the right hand side of \eqref{equicond}, we set
	\begin{equation}\label{y3bary3}
		\begin{split}
			Y_3(\sigma)&:= \frac{3}{2}\abs{O_{3} }+2\abs{O_{4,2} }
			+\frac{5}{2}\abs{O_{5,1}}
			+3\abs{O_{5,2}}+4\abs{O_6}
			+5\abs{O_2}+ 6\abs{O_{4,1} },\\
			\bar Y_3(\sigma)&:=\frac{3}{2}\abs{\bar O_{3} }+2\abs{\bar O_{4,2} }
			+\frac{5}{2}\abs{\bar O_2}
			+3\abs{\bar O_{5,2}}+4\abs{\bar O_6}+5\abs{\bar O_{5,1}}+ 6\abs{\bar O_{4,1} },
		\end{split}
	\end{equation}
	and let 
	\begin{equation}\label{y'3def}
		Y'_3(\sigma):=Y_3(\sigma)-\bar Y_3(\sigma).
	\end{equation}
	% Here we have arranged the terms in the ascending order of their coefficients.  
	Repeating the analysis that leads to \eqref{equicond}, we find that \eqref{ord2bdd} 
	is equivalent to 
	\begin{equation}\label{x3-}
		Y_3(\sigma)	-\abs{O_1}>5\kappa(Y_1(\sigma),Y_2(\sigma)).
	\end{equation}
	The proofs  for \eqref{equicond} show that $	Y_3(\sigma)	-\abs{O_1}>0$ always holds. Therefore, 
	we only need to consider the situations where $\kappa(Y_1(\sigma),Y_2(\sigma))>0$, which,	by  Table \ref{table:kab} above, include the cases where $Y_1(\sigma)=0$ and $
	Y_2(\sigma)\in \{2,3\}$; and also both $Y_1(\sigma), Y_2(\sigma)\in \{1,2\}$. 
	
	To  prove \eqref{x3-}, we rewrite it as
	\begin{equation}\label{twoparts}
		\left(Y_3'(\sigma)-\abs{O_1'}\right)+\left(\bar Y_3(\sigma)-\abs{\bar O_1}-5\kappa(Y_1(\sigma),Y_2(\sigma))\right)
		>0.
	\end{equation}
	The first part $Y_3'(\sigma)-\abs{O_1'}$ in \eqref{twoparts} can be controlled by the following 
	\begin{lem}\label{lem:o1'y3'} We always have
		\begin{equation}\label{o1new-1}
			\abs{O'_1}\leq Y_3'(\sigma).
		\end{equation}
		Moreover, if $\abs{\bar O_1}+Y_1(\sigma)>0$ and $\abs{\bar O_1}+Y_1(\sigma)+Y_2(\sigma)$ is odd, then
		\begin{equation}\label{o1new0}
			\abs{O'_1}< Y'_3(\sigma).
		\end{equation}
	\end{lem}
	% We remark that the equality in \eqref{o1new-1} may hold in some cases. Indeed, consider the example where $\sigma \in \mathcal{P}_{16}$ such that $(\sigma_{2i-1}, \sigma_{2i})_{1\leq i\leq 8}$ is given by
	%	$$
	%	(1,9), (2,3),(4,8),(5,13), (6,7), (10,12), (11,15), (14,16).
	%	$$
	%	Then one can compute that $\abs{O_1'}=Y'_3(\sigma)=2$. 
	
	\begin{proof}
		Define a set $
		\widehat{O}_{1,NEW}
		$ so that it equals $\widehat{O}_1$ (defined in \eqref{hato1def}) if $(1,2)\notin O_1$, and $\widehat{O}_1\backslash
		\{0,1\}$ otherwise. We have the relation 
		\begin{equation}\label{o1new1}
			\abs{\widehat{O}_{1,NEW}}=2\abs{O_1'}.
		\end{equation}
		Moreover, 
		\begin{equation}\label{o1new2}
			\mbox{if }\abs{\bar O_1}+Y_1(\sigma)=\sum_{i=1}^3 \abs{\bar O_i}>0, \quad \mbox{then }    \abs{{\widehat{O}_{1,NEW}} \cap \{1,2\}}\leq 1.
		\end{equation}
		Recall the definition of the function $h_{\sigma}$ in \eqref{hsigmadef}. Similarly to \eqref{hnsum} and \eqref{hnsum2}, we have
		\begin{equation}\label{o1new3}
			\begin{split}
				\sum_{3\leq \ell \leq k} h_{\sigma}(\ell)&\geq    \sum_{\ell \in \widehat{O}_{1,NEW}\cap [3, k] } h_{\sigma}(\ell)  \geq 2\abs{\widehat{O}_{1,NEW}\cap [3,k]}\\
				&=2\left( \abs{\widehat{O}_{1,NEW}}-\abs{{\widehat{O}_{1,NEW}} \cap \{1,2\}}\right),
			\end{split}
		\end{equation}
		and
		\begin{equation}\label{o1new4}
			\begin{split}
				\sum_{3\leq \ell \leq k} h_{\sigma}(\ell)&=\sum_{\lceil  i/4 \rceil \leq 2< \lceil j/4\rceil} \mathbf{\Lambda}_{4k}(i,j)+
				2      \sum_{3\leq \lceil  i/4 \rceil < \lceil j/4\rceil} \mathbf{\Lambda}_{4k}(i,j)\\
				& =2\left(\sum_{\lceil  i/4 \rceil \leq 2< \lceil j/4\rceil} \mathbf{\Lambda}_{4k}(i,j)+
				\sum_{3\leq \lceil  i/4 \rceil < \lceil j/4\rceil} \mathbf{\Lambda}_{4k}(i,j)\right)-\sum_{\lceil  i/4 \rceil \leq 2< \lceil j/4\rceil} \mathbf{\Lambda}_{4k}(i,j)\\
				&\leq 2\sum_{i=1}^6 \abs{O_i'}-\abs{{\widehat{O}_{1,NEW}} \cap \{1,2\}}.
			\end{split}
		\end{equation}
		Combining \eqref{o1new1}, \eqref{o1new3} and \eqref{o1new4}, we get
		\begin{equation}
			\abs{O_1'}\leq \sum_{i=2}^6 \abs{O_i'}+\frac{1}{2} \abs{{\widehat{O}_{1,NEW}} \cap \{1,2\}}.
		\end{equation}
		Thus in any case, we have
		\begin{equation}\label{o1new5}
			\abs{O_1'}\leq \sum_{i=2}^6 \abs{O_i'}+1.
		\end{equation}
		Furthermore, if $\sum_{i=1}^3\abs{\bar O_i}>0$, then by \eqref{o1new2}, we get
		\begin{equation}\label{o1new6}
			\abs{O_1'}\leq \sum_{i=2}^6 \abs{O_i'}.
		\end{equation}
		Equation		 \eqref{o1new0} now follows directly from \eqref{y'3def}, \eqref{o1new6} and the implication
		$$
		\abs{\bar O_1}+Y_1(\sigma)+Y_2(\sigma) \mbox{ is odd}
		\Rightarrow \sum_{i=1}^6 \abs{O_i'}>0,
		$$
		which is due to the evenness of $h_{\sigma}(1)$ and $h_{\sigma}(2)$  (see the line below \eqref{hsigmadef}).
		
		To prove \eqref{o1new-1}, we note that all  coefficients in  front of $\abs{O_{i}}$ and $\abs{\bar O_i}$ in the right hand side of \eqref{y3bary3} are at least 2, except for $\abs{O'_3}$. Hence, the only possible way  to make
		$Y_3'(\sigma)<\abs{O_1'}$ is by requiring $\abs{O_1'}=2$, $\abs{O_3'}=1$, and $\abs{O'_i}=0$ for $i=2,4,5,6$. This implies that, there  exists $i_0>2$ such that $h_{\sigma}(i_0)=1$, which again contradicts with the property that $h_{\sigma}(\ell)$ is even for all $\ell$. Therefore, we prove  \eqref{o1new-1}.

	\end{proof}
	The second part of  the right hand side  of \eqref{twoparts} is controlled by the following  Lemma \ref{lem:y1bar}. Combining this result and  Lemma \ref{lem:o1'y3'},
	we deduce \eqref{twoparts} and thus   \eqref{ord2bdd}. Therefore, we finish \eqref{2error}. 
	
	\begin{lem}\label{lem:y1bar}
		If   $\abs{\bar O_1}+Y_1(\sigma)>0$ and $\abs{\bar O_1}+Y_1(\sigma)+Y_2(\sigma)$ is odd, then we have
		\begin{equation}\label{bary3-o2}
			\bar Y_3(\sigma)-\abs{\bar O_1}\geq 5\kappa(Y_1(\sigma),Y_2(\sigma)).
		\end{equation}
		Otherwise,     we have 
		\begin{equation}\label{bary3-o1}
			\bar Y_3(\sigma) -\abs{\bar O_1}>  5\kappa(Y_1(\sigma),Y_2(\sigma)).
		\end{equation}
	\end{lem}
	\begin{proof}%[Proof of Lemma \ref{lem:y1bar}]
		We shall only give a detailed proof of \eqref{bary3-o2}. The argument for \eqref{bary3-o1} is similar. We first rewrite \eqref{bary3-o2} in a more convenient form. Let $Z$ be a  $4\times 4$ matrix   given by
		\begin{align}\label{defz}
			Z=
			\begin{bmatrix}
				0 & 5/2 & 3/2 & 
				3/2 \\
				5/2 & 6 & 5& 5 \\
				3/2 & 5 & 2& 3 \\
				3/2 & 5 & 3& 4 
			\end{bmatrix}.
		\end{align}
		We observe that
		\begin{equation}\label{minz}
			\min\left\{Z(i,j):i=1<j \mbox{ or }j=1<i\right\}=\frac{3}{2},\,\, \min\left\{Z(i,j):i,j\geq 2\right\}=2.
		\end{equation}
		Let $\mathcal{X}_0$ be the set of $4\times 4$ doubly sub-stochastic matrix with entries being 0 or 1,
		$$
		\mathcal{X}_0=\big\{  X\in \{0,1\}^{4\times 4}: \forall \, 1\leq j\leq 4, \sum_{i'=1}^4 X(i',j)\leq 1; \forall \, 1\leq i\leq 4, \sum_{j'=1}^4 X(i,j')\leq 1\big\}.
		$$
		Given three integer $a_1,a_2,a_3$, we define $\mathcal{X}(a_1,a_2,a_3)$ to be the set
		$$
		\big\{X\in \mathcal{X}_0:  X(1,1)=a_1,\, \sum_{i=2}^4 X(1,i)+\sum_{i=2}^4 X(i,1)=a_2, \, \sum_{i=2}^4 \sum_{j=2}^4 X(i,j)=a_3  \big\}.
		$$
		Clearly, given the values of $\abs{\bar O_1}, Y_1(\sigma), Y_2(\sigma)$,   the $4\times4$ matrix $(\mathbf{\Lambda}_{4k+4}(i,j))_{1\leq i\leq 4, 5\leq j\leq 8}$ has to belong to the set $\mathcal{X}(\abs{\bar O_1}, Y_1(\sigma), Y_2(\sigma))$. 
		Hence, by the definition of $\bar Y_3(\sigma)$,  to prove \eqref{bary3-o2}, it suffices to prove  
		\begin{equation}\label{zx}
			\min_{X\in \mathcal{X}(\abs{\bar O_1}, Y_1(\sigma), Y_2(\sigma))} \sum_{1\leq i,j\leq 4} Z(i,j)X(i,j) \geq \abs{\bar O_1}+\kappa(Y_1(\sigma)+Y_2(\sigma)).
		\end{equation}
		The condition  `$\abs{\bar O_1}+Y_1(\sigma)>0$ and $\abs{\bar O_1}+Y_1(\sigma)+Y_2(\sigma)$ is odd'
		includes three cases.
		
		Case \textcircled{1}: $\abs{\bar O_1}=0$, $Y_1(\sigma)=1$ and $Y_2(\sigma)=2$. Since $\kappa(1,2)=3/2$, \eqref{zx}  becomes
		\begin{equation}\label{zx2}
			\min_{X\in \mathcal{X}(0,1,2) } \sum_{1\leq i,j\leq 4} Z(i,j)X(i,j) \geq 15/2. 
		\end{equation}
		The condition $X\in \mathcal{X}(0,1,2)$ implies 
		$$\sum_{i=2}^4 \left(X(i,1)+X(1,i)\right)=1,$$ and that there are two 
		distinct pairs $(i_1,j_1), (i_2,j_2)\in [2,4]\times [2,4]$ such that $X(i_1,j_1)=X(i_2,j_2)=1$. 
		We consider the following two scenarios. 
		\begin{itemize}
			\item If $X(3,3)=1$, say,  $(i_1,j_1)=1$, then we must have $(i_2,j_2) \neq (3,4)$ or $(4,3)$. Hence $Z(i_2,j_2)\geq 4$.
			Consequently, by \eqref{minz},
			\begin{equation*}
				\begin{split}
					\sum_{1\leq i,j\leq 4} Z(i,j)X(i,j)
					&\geq \frac{3}{2} \left( \sum_{i=2}^4 \left(X(i,1)+X(1,i)\right) \right) +Z(3,3)+Z(i_2,j_2)\\
					&\geq \frac{3}{2}+2+4=\frac{15}{2}.
				\end{split}
			\end{equation*}
			\item If $X(3,3)=0$, then both $(i_1,j_1)$ and $(i_2,j_2)$ are not equal to $(3,3)$. Thus $Z(i_1,j_1), Z(i_2,j_2)\geq 3$, which implies that, by \eqref{minz},
			$$
			\sum_{1\leq i,j\leq 4} Z(i,j)X(i,j) \geq \frac{3}{2}+3+3=\frac{15}{2}.
			$$
			
		\end{itemize}Hence we have verified \eqref{zx2} in Case \textcircled{1}. 
		
		Case \textcircled{2}: $\abs{\bar O_1}=0$, $Y_1(\sigma)=2$ and $Y_2(\sigma)=1$. Then $\kappa(Y_1(\sigma),Y_2(\sigma))=1$. Using \eqref{minz},
		$$
		\sum_{1\leq i,j\leq 4} Z(i,j)X(i,j) \geq \frac{3}{2}Y_1(\sigma)+2Y_2(\sigma)=5=5\kappa(Y_1(\sigma),Y_2(\sigma)),
		$$
		proving \eqref{zx} in this case. 
		
		Case \textcircled{3}: $\abs{\bar O_1}=1$, $Y_1(\sigma)=0$ and $Y_2(\sigma)=2$. Then $\kappa(Y_1(\sigma),Y_2(\sigma))=1$. From the arguments in Case \textcircled{1}, we see that
		$$
		\min_{X\in \mathcal{X}(1,0,2) } \sum_{1\leq i,j\leq 4} Z(i,j)X(i,j)
		\geq \min\{2+4,3+3\}=6=\abs{\bar O_1}+5\kappa(Y_1(\sigma),Y_2(\sigma)).
		$$

		We have verified \eqref{zx} in all possible cases, and thus we deduce \eqref{bary3-o2}. The other statement \eqref{bary3-o1} can be proved similarly. %This finish Lemma \ref{lem:y1bar}.

	\end{proof}

	\section{The case of GOE }\label{sec:pfgoe}
	
	\subsection{Pfaffian structure of GOE } GOE is a Gaussian measure on the space of Hermitian matrices. The joint density of eigenvalues of GOE is given by \eqref{jpdf of Gauss ensemble} with $\beta=1$. In this section,  we use a superscript (1) to denote various quantities regarding GOE.
	The $k$-point correlation function of eigenvalues of GOE is \cite{M}
	\begin{align}
		\rho^{(1)}_{k,n}(\lambda_1,...,\lambda_k) = \pfa\left(JK_n^{(1)}(\lambda_i,\lambda_j)\right)_{1\leq i,j\leq k},
	\end{align}
	where 
	\begin{align}
		\begin{aligned}
			J = \begin{bmatrix}
				0  & 1 \\
				-1 & 0
			\end{bmatrix},\quad 
			K_n^{(1)}(x, y) 
			=\begin{bmatrix}
				S^{(1)}_{n}(x, y)+\alpha_n(x) & V^{(1)}_{n}(x, y) \\
				J^{(1)}_{n}(x, y) & S^{(1)}_{n}(y, x)+\alpha_n(y)
			\end{bmatrix}. 
		\end{aligned}
	\end{align}
	Here, 
	\begin{align*}
		\alpha_n(x)=&
		\left\{
		\begin{aligned}
			&\sqrt{2m+1}\varphi_{2 m}\left(\sqrt{2m+1}x\right) \big / \int_{-\infty}^{\infty} \varphi_{2 m}(t) \mathrm{d} t,& &\quad \text{if } n=2m+1,\\
			&0,& &\quad \text{if } n=2m,
		\end{aligned}
		\right.
	\end{align*}
	and
	\begin{align*}
		\begin{aligned}
			S_{n}^{(1)}(x, y) & =\sqrt{n}
			\sum_{j=0}^{n-1} \varphi_{j}\left(\sqrt{n}x\right)\varphi_{j}\left(\sqrt{n}y\right)+\frac{n}{2} \varphi_{n-1}\left(\sqrt{n}x\right) \int_{-\infty}^{\infty} \varepsilon(\sqrt{n}y-t) \varphi_{n}(t) \mathrm{d} t, \\
			V_{n}^{(1)}(x, y) & =-\frac{\partial}{\partial y} S_{n}^{(1)}(x, y), \\
			J_{n}^{(1)}(x, y) & =\int_{-\infty}^{\infty} \varepsilon(x-t) S^{(1)}_{n}(t, y) \mathrm{d} t-\varepsilon(x-y)+
			\gamma_n(x)-\gamma_n(y),
		\end{aligned}
	\end{align*}
	where  $$
	\gamma_n(x)  =\int_{-\infty}^{\infty} \varepsilon(x-t) \alpha_n(t) \mathrm{d} t,
	$$ and
	$ \varepsilon(x)=(1 / 2) \operatorname{sign}(x)$  is  $1 / 2$, $-1 / 2$  or $0$ according as $ x>0$, $x<0 $ or $ x= 0$,

	Let $\widehat{K}^{(1)}_n$ be the rescaled version of $K^{(1)}_n$ defined as in \eqref{rescaledkernel}. 
	As in the proof of Lemma \ref{Convergence of the rescaling kernel}, using the result in 
	\cite{KS}, we get
	\begin{align}
		\label{estimate of the GOE kernel}
		\abs{ 		\widehat{K}^{(1)}_{n}\left(x_0+\frac{x}{n \rho_{\semi}(x_0)},x_0+\frac{y}{n\rho_{\semi}(x_0)}\right)-K^{(1)}(x, y) }   \lesssim     \begin{bmatrix}
			n^{-1/2}  & n^{-1} \\
			n^{-1/2}   & n^{-1/2}
		\end{bmatrix}
	\end{align}
	uniformly for $\abs{x},\abs{y}\lesssim n^{-1/2}$. Here,  the limiting kernel 
	\begin{align}
		K^{(1)}(x,y) =& 
		\begin{bmatrix}
			K_{\sin}(x-y) & {\partial_x}K_{\sin}(x-y) \\
			\int_{0}^{x-y}K_{\sin}(t)\mathrm{d}t-\varepsilon(x-y) & K_{\sin}(x-y)
		\end{bmatrix}.\label{K1}
	\end{align} 
	In addition to the local scaling limit for $K_n^{(1)}$, we can also obtain global bounds for $K^{(1)}_n$ and its derivatives. Indeed, one can check that Lemma \ref{order of multi kernel} still holds true if we replace the GSE kernels $S_n, V_n,J_n$ with GOE kernels $S_n^{(1)}+\alpha_n, V^{(1)}_n, J_n^{(1)}$, respectively. We do not repeat the details.  
	
	\subsection{Estimates of Pfaffians}
	We follow the proof outline given in Section \ref{sec:pfoutline}. Let $A^{(1)}_n:=n^{-3/2}A$, $A\subset(0,+\infty)$. 
	As in the proof of GSE case, we separate  $L_{k,n}^{(1)}$ into the diagonal part $L_{k,n,D}^{(1)}$ and the off-diagonal part $L_{k,n,O}^{(1)}$. Again, we denote by $\M^{(1)}_{4k}$ the $4k \times 4k$ correlation matrix so that
	$$
	\rho^{(1)}_{2k,n}(\lambda_1,x_1,\ldots, \lambda_k,x_k)=\pfa \M^{(1)}_{4k}.
	$$ Let $M^{(1)}_{i,j}$ be the $(i,j)$-th block of $\M^{(1)}_{4k}$ with size $4\times 4$.  Using the estimate \eqref{K1}, we have  the uniform estimate for the diagonal part:	\begin{equation}\label{l1kn} 	
		\begin{split}
			& L_{k,n,D}^{(1)}(\lambda,\ldots, \lambda_k) \\
			=&\prod_{i=1}^{k}\int_{\lambda_i+A_n^{(1)}}\pfa  M_{i,i}^{(1)}\mathrm{d}x_i
			\\
			=& \prod_{i=1}^{k}\int_{n\rho_{\semi}(\lambda_i)A_n^{(1)}} n\rho_{\semi}(\lambda_i)\left(1-K^2_{\sin}(u) + {\partial_u}K_{\sin}(u)\left(\int^{u}_0K_{\sin}(t)\mathrm{d}t +\frac{1}{2}\right)\right)\mathrm{d}u\\
			&+O\left(n^{-1/2}\right)\\
			=& \prod_{i=1}^{k}\int_{n^{-1/2}\rho_{\semi}(\lambda_i)A} n\rho_{\semi}(\lambda_i)\left(
			\dfrac{\pi^2u}{6}+O\left(u^2\right)\right)\mathrm{d}u+O\left(n^{-1/2}\right) \\
			=&\left(\int_{A} \dfrac{u}{48\pi}\mathrm{d}u\right)^k \prod_{i=1}^{k} \left( 2\pi\rho_{\semi}(\lambda_i)\right)^3+O\left(n^{-1/2}\right).
		\end{split}
	\end{equation}
	To control the off-diagonal part, we have to make two rounds of transformations to $\M^{(1)}_{4k}$, similar to those  for GSE, which eventually change $\M^{(1)}_{4k}$ to  $\M^{(1,F)}_{4k}$.
	The only difference  is  that we need to replace the factors $n^{6/5}$  and $n^{-6/5}$ with $n^{3/2}$ and $n^{-3/2}$, respectively, 
	in the definition \eqref{operation of pfaffian}.
	Ultimately, we can bound $\M^{(1,F)}_{4k}$ from above by $\U^{(1)}_{4k}$ in the entrywise sense as in Lemma \ref{lem:round1}, while this time the diagonal block $U_{i,i}$ and off-diagonal block $U_{i,j}$ ($j>i$) are  given by 
	\begin{align}\label{order of GOE2}
		U_{i,i}^{(1)} =	
		\begin{bmatrix}
			0 & n &  n^{3/2} & 1 \\
			n &0 &  1 &  1\\
			n^{3/2} & 1 & 0& 1 \\
			1& 1& 1 &0
		\end{bmatrix}, \quad 
		U_{i,j}^{(1)} =	
		\begin{bmatrix}
			n^{3/2}\log n & n^{1/2} &  1 & 1 \\
			n^{1/2} & n^{-1/2} &  n^{-1} &  n^{-1}\\
			1 & n^{-1} & n^{-1/2} & n^{-1} \\
			1& n^{-1} & n^{-1} & n^{-3/2}
		\end{bmatrix}.
	\end{align}
	Define $U^{(1)}_{4k}(\sigma)$ similarly to \eqref{def of U} and  consider the set partitions given in Figure \ref{fig1} and Figure \ref{fig2}. According to the growth order of $n$ in $U_{i,i}^{(1)}$, we need to further decompose 
	\begin{equation}
		D_2:=D_{2,1}\cup D_{2,2},
	\end{equation}
	where
	\begin{equation*}
		\begin{split}
			D_{2,1}:=&\left\{(i,j)\in [4k]\times [4k]:
			\lceil i/4\rceil = 
			\lceil j/4\rceil, \mathbf{\Lambda}_{4k}^{(1)}(i,j)=1, i\equiv 1(\mod 4),\,j\equiv 3(\mod 4)\right\},\\
			D_{2,2}:=&\left\{(i,j)\in [4k]\times [4k]:
			\lceil i/4\rceil = 
			\lceil j/4\rceil, \mathbf{\Lambda}_{4k}^{(1)}(i,j)=1, i\equiv 1(\mod 4),\,j\equiv 4(\mod 4)\right\}.
		\end{split} 
	\end{equation*}
	By the estimates 
	\eqref{order of GOE2}, we get
	\begin{equation}
		U^{(1)}_{4k}(\sigma)\lesssim \left(\log^C n\right)  n^{\ord_4}, 
	\end{equation}
	where 
	\begin{equation}
		\ord_4= |D_1|+\frac{3}{2}|D_{2,1}|+\frac{3}{2} |O_1|+\frac{1}{2} |O_2|-\frac{1}{2}|O_4|-|O_5|-\frac{3}{2}|O_6|.        
	\end{equation}
	The constraint $2|O_1|+|O_2|+|O_{3} |+|D_1|+|D_2|=k$  thus implies that
	$\ord_4\leq 3k/2$. Moreover, the equality holds if and only if $\abs{D_{2,1}}=k$ and $\abs{O_i}=0$ for all $1\leq i\leq 6$.
	In particular, we have $$
	\ord_4<\frac{3k}{2}, \quad \forall\, \sigma \in \mathcal{P}_{4k}\backslash \mathcal{D}_{4k}.
	$$
	This implies that $L^{(1)}_{k,n,O}\to 0$ as $n\to\infty$, as desired. Analysis for $E^{(1)}_{1,k,n}$ and $E^{(2)}_{2,k,n}$ also follows from similar reasoning to those given for  GSE, and we omit further details.  Combining these estimates and recalling 
	\eqref{l1kn}, we complete the proof of Theorem \ref{small gaps of GbetaE} in the GOE case ($\beta$=1).

\end{document}